\documentclass[a4paper]{amsart}
\usepackage[all]{xy}
\usepackage{enumerate}
\usepackage{amssymb}
\usepackage{amsmath, amsthm, color, tikz}
\usepackage[sans]{dsfont}
\usepackage[utf8]{inputenc}

\DeclareMathOperator{\id}{id}

\newtheorem{thm}{Theorem}[section]

\newtheorem{prop}[thm]{Proposition}
\newtheorem{lem}[thm]{Lemma}

\theoremstyle{definition}
\newtheorem{defn}[thm]{Definition}
\newtheorem{ex}[thm]{Example}

\theoremstyle{remark}
\newtheorem{rem}[thm]{Remark}

\newcommand{\kk}{\mathbb K}

\newcommand{\brr}[1]{\left[#1\right]}

\newcommand{\ds}{\displaystyle}

\newcommand{\al}{\alpha}
\newcommand{\be}{\beta}

\newcommand{\A}{A}                      
\renewcommand{\gg}{\mathfrak{g}}        
\newcommand{\hh}{\mathfrak{h}}          

\begin{document}

\title{Universal Algebra of a Hom-Lie Algebra and group-like elements}

\author{Camille Laurent-Gengoux, Abdenacer Makhlouf and Joana Teles }
\address{\newline C. Laurent-Gengoux : Institut Elie Cartan de Lorraine, UMR 7122,
Universit\'e de Lorraine,
57045 Metz, France  \newline  A. Makhlouf : Laboratoire de Math\'ematique, Informatique et Applications, Universit\'e de Haute-Alsace, 68093 Mulhouse, France \newline J. Teles : CMUC, Departament of Mathematics, University of Coimbra, P-3001-454 Coimbra, Portugal}
\thanks{The first and last authors were supported by the Centro de Matemática da
Universidade de Coimbra (CMUC), funded by the European Regional
Development Fund through the program COMPETE and by the Portuguese
Government through the FCT - Fundação para a Ciência e a Tecnologia
under the project PEst-C/MAT/UI0081/2013. }

\begin{abstract}
We construct the universal enveloping algebra of a Hom-Lie algebra and endow it with a Hom-Hopf algebra structure. We discuss group-like elements that we see as a Hom-group integrating the initial Hom-Lie algebra.
\end{abstract}

\maketitle

\tableofcontents

\section*{Introduction}

Hom-Lie algebras are  a generalization of  Lie algebras, formalizing an algebraic structure which appeared first in quantum deformations of Witt and Virasoro algebras, see for example \cite{AizawaSaito,ChaiIsKuLuk,CurtrZachos1,Kassel1,Hu}. A quantum deformation or a $q$-deformation of an algebra of vector fields is obtained when replacing a usual derivation by a $\sigma$-derivation $d_\sigma$ that satisfies a twisted Leibniz rule $d_\sigma(f g)=d_\sigma(f)g+\sigma(f) d_\sigma(g) $, where $\sigma$ is an algebra endomorphism of a commutative associative algebra. An example of  a $\sigma$-derivation is the Jackson derivative on polynomials in one variable. A general construction of quasi-Lie algebras and the introduction of Hom-Lie algebras were given in \cite{HLS}. The corresponding  associative algebras, called Hom-associative algebras, were introduced in \cite{MS}, where it is shown that they are Hom-Lie admissible,  while the enveloping algebra of a Hom-Lie algebra was constructed in \cite{Yau08}. Moreover, Hom-bialgebras and Hom-Hopf algebras were studied in \cite{MS2009,MS2010a,Yau4}. Further results could be found in \cite{AM,AEM,BEM,E,LS,MS2,MS2010,Sh,Yau09,Yau10}.

The purpose of this article is two-fold. First, we construct explicitly the universal enveloping algebra of a multiplicative Hom-Lie algebra and show that it is a Hom-Hopf algebra. Then, we intend to mimic the construction of a Lie group integrating a Lie algebra $\gg$ obtained by choosing, as a candidate for integrating the Lie algebra $\gg$, group-like elements of the universal enveloping algebra $ {\mathcal U}\gg$.

This task, in particular the second step, is not trivial, and forced us to reconsider the way antipodes are generally defined on bialgebras, as well as the definition of invertible elements on Hom-groups.  Still, we are able to present a Hom-group that we claim to be the integration of a Hom-Lie algebra, but it is more involved than simply group-like elements in the  universal enveloping algebra of a Hom-Lie algebra.
Before describing our construction step by step, let us discuss what inverse means in the context of Hom-algebras.

\subsection*{Invertibility and inverse on Hom-algebras or Hom-groups}

There is one natural manner, which was already considered by \cite{Yau08}, to construct the universal enveloping algebra $ {\mathcal U}\gg$ of a multiplicative Hom-Lie algebra $\gg$. This object is, as expected, a Hom-bialgebra. But, as we shall see while  proving Theorem \ref{thm:HopfAlgOnTrees}, it turns out \emph{not} to be a Hom-Hopf algebra in the sense of \cite{MS2010a}, because the antipode is not an inverse of the identity map for the convolution product.

This failure is, however, productive, in the sense that it paves the way for what seems to be a definition of a Hom-Hopf algebra suitable in this context. Let us explain the situation. In the Hom-bialgebra $ ({\mathcal U}\gg,\vee, \alpha,\Delta, \eta, \epsilon)$  the usual  antipode $S$ (for Hopf algebra structure) does not satisfy the axiom:
$$ S \star id = id \star S =   \eta \circ \epsilon $$
with $\star $ being the convolution product, defined by $S \star T = \vee \circ (S \otimes T) \circ \Delta $
for two arbitrary linear endomorphisms $S,T$  of ${\mathcal U}\gg$, as in the classical case. The antipode $S$ only satisfies a weakened condition:    for any $ x \in  {\mathcal U}\gg$, there exists an integer $ k  \in {\mathbb N}$ such that:
   \begin{equation}\label{eq:weaker} \al^k \circ \vee \circ (S \otimes id) \circ \Delta \, (x)  = \al^k \circ \vee \circ (id \otimes S) \circ \Delta \, (x) = \eta \circ  \epsilon\, (x)  .
			\end{equation}
The smallest such integer $k$ is called the invertibility index of $x$.
We define Hom-Hopf algebras as bialgebras satisfying this weakened condition.

This definition is indeed not surprising. Given a Hom-associative algebra $ ( {\A} , \vee , \al )$  that admits a unit ${\mathds{1}} $, it is tempting to define invertible elements as being elements $x$ in $ {\A}$  such that there exists $ y \in {\A}$ with $ x \vee y=y \vee x= {\mathds{1}}$.
Alternatively, it may be tempting to define invertible elements to be those elements $x$ in $ {\A}$ such that there exists $ y \in {\A}$ with $ \al(x \vee y)=\al(y \vee x)= {\mathds{1}}$
as in \cite{F}. However, there is an issue with both definitions: invertible elements in any of these two senses are in general not stable under $\vee $.

In order to get a notion of invertible elements that would allow those to be invariant under $\vee $, we  say that an element $x \in {\A}$ is \textbf{hom-invertible} if and only if there exists $y \in {\A}$ (not necessarily unique) called a \textbf{hom-inverse} and an integer $k \in {\mathbb N}$ such that
$$ \al^k ( x \vee y ) = \al^k ( y \vee x ) = {\mathds{1}} .$$
This definition is consistent with Equation (\ref{eq:weaker}). The antipode $S$ that we have constructed on the Hom-bialgebra $ ({\mathcal U}\gg,\vee, \Delta, \eta, \epsilon)$ becomes now a kind of hom-inverse of the identity for the convolution product, making $ ({\mathcal U}\gg,\vee, \Delta, \eta, \epsilon, S)$ a Hom-Hopf algebra.

This issue being solved, we intend to find a Hom-group integrating a Hom-Lie algebra. Having modified the definition of a hom-inverse, we have to modify the definition of Hom-Lie group accordingly.

\begin{defn}\label{def:ourHomGroup}
A \textbf{Hom-group} is a set $(G,\vee,\al,{\mathds{1}})$ equipped with a Hom-associative product with unit ${\mathds{1}} $ and an anti-morphism $g \to g^{-1}$ such that, for any $g \in G$, there exists an integer $k \in {\mathbb N}$ satisfying
$$ \al^k ( g\vee g^{-1}) = \al^k ( g^{-1} \vee g ) = {\mathds{1}} .$$
The smallest such integer $k$ is called the invertibility index of $g$.
\end{defn}

\subsection*{From Hom-Lie algebras to Hom-groups}

Group-like elements in a Hom-Hopf algebra $ {\A}$ (equipped with an antipode satisfying the weakened assumption (\ref{eq:weaker})) form a Hom-group (group-like elements being defined as formal series $ g(\nu )\in  {\A}[[\nu]]$
with some assumptions). It is therefore tempting to define the object integrating the Hom-Lie algebra $ \gg$ as being the set of group-like elements in the Hom-Hopf algebra $ {\mathcal U}\gg[[\nu]]$. However, this definition is irrelevant: there is in general very little group-like elements in $ {\mathcal U}\gg[[\nu]]$, except for the  unit $ {\mathds{1}}$ itself.

It is possible, fortunately, to go around this difficulty by defining, for all $p \in {\mathbb N}$, {$p$-order formal group-like elements} as being elements in $ {\mathcal U}\gg[[\nu]]$  satisfying
$ g(0)={\mathds{1}} $ and:
 $$ \Delta  g(\nu) =  g(\nu) \otimes  g(\nu) \hspace{1cm} \left[\nu^{p+1}\right]$$
(where $\left[\nu^{p+1}\right]$ means "modulo $ \nu^{p+1} $").
It is routine to check that $p$-order formal group-like element do form a Hom-group, with inverse given by the antipode.
Then we consider sequences $ (g_p(\nu))_{p \geq 0}$, with $g_p(\nu)$ a $p$-order formal group-like element,
such that the quotient of $ g_{p+1} (\nu)$ modulo $ \nu^{p+1}$ is $\al ( g_p(\nu) )  $  for all $p \in {\mathbb N}$.
 We call these sequences {formal group-like sequences} when their invertibility index is bounded. We show that formal group-like sequences do form a Hom-group, with inverse again induced by the antipode. Moreover an exponential map valued in formal group-like sequence can be constructed making this Hom-group a reasonable candidate for being considered as a Hom-group integrating the Hom-Lie algebra $ \gg$, as several theorems will show at the end of this article.


\section{Hom-Lie algebras, Hom-associative algebras and Hom-bialgebras}

Given $\gg$ a vector space and a bilinear map $ \brr{\, , \, }: \gg \otimes \gg \to \gg$, by \textbf{endomorphism of
$(\gg,\brr{\, , \, } )$}, we mean a linear map $\al: \gg \to \gg$ such that
$$
\al (\brr{x,y}) = \brr{\al (x), \al (y)}
$$
for all $x,y \in \gg$. We now define Hom-Lie algebras, sometimes called multiplicative Hom-Lie algebras.

\begin{defn}\label{def:hom:Lie:algebra} \cite{HLS}
A \textbf{Hom-Lie algebra} is a triple  $(\gg, \brr{\, , \, }, \al)$ with $\gg$ a vector space equipped with a
skew-symmetric bilinear map $ \brr{\, , \, }:\gg \otimes \gg \to \gg$ and an endomorphism
$\al$ of $(\gg,\brr{\, , \, })$ such that:
\begin{equation}
\brr{\al (x),\brr{y,z}}+\brr{\al (y),\brr{z,x}}+\brr{\al
(z),\brr{x,y}}=0, \quad \forall x,y,z\in\gg  \quad \hbox{(Hom-Jacobi identity)}. \label{eq:hom:Jacobi:algebra}
\end{equation}
A \textbf{morphism} between Hom-Lie algebras $(\gg,\brr{\, , \, }_\gg,\al)$ and $(\hh,\brr{\, , \, }_\hh,\be)$ is a linear map $\psi:\gg \to \hh$
such that $\ds \psi(\brr{x,y}_\gg)=\brr{\psi( x),\psi (y)}_\hh$ and $\psi (\al(x))=\be(\psi(x))$
for all $x,y \in \gg$.
When $\hh$ is a vector subspace of $\gg$ and $\psi$ is the inclusion map, one speaks of \textbf{Hom-Lie subalgebra}.
\end{defn}

\begin{rem}\label{Rmk:IdealAndInverse}
Let $(\gg, \brr{\, , \, }, \al)$ be a Hom-Lie algebra. The subspace ${\mathfrak k} \subset \gg$ of elements $x \in \gg$ such that there exists an integer $k$ with $\al^k (x)=0$ is a Hom-Lie subalgebra. It is even a Hom-Lie ideal, i.e. the quotient $\gg/{\mathfrak k} $ inherits a structure of Hom-Lie algebra. By construction, its induced morphism $\underline{\al}: \gg/{\mathfrak k} \to \gg/{\mathfrak k}$ is invertible. The induced bracket $\underline{\al}^{-1} \circ \underline{\brr{\, , \, }}$ is therefore a Lie algebra bracket, see \cite{G}. Also, the natural projection $\gg/{\mathfrak k} $ is a morphism of Hom-Lie algebra.
\end{rem}

\begin{defn} \label{hom-associative}\cite{MS}
A \textbf{Hom-associative algebra} is a triple $(\A, \mu, \al)$ consisting of a  vector space $\A$, a bilinear map $\mu: \A \otimes \A \to \A$ and an endomorphism $\al$ of $( \A, \mu) $ satisfying
$$
\mu(\al(x), \mu(y,z))= \mu (\mu(x,y), \al (z)), \quad \forall x,y,z \in \A \hbox{ (Hom-associativity)}.
$$
A Hom-associative algebra $(\A, \mu, \al)$ is called \textbf{unital} if there exists a linear map $\eta: \kk \to \A$  such that
$$
\mu \circ (id_{\A} \otimes \eta) = \mu \circ (\eta \otimes id_{\A}  )=  \al \hbox{ and } \al \circ \eta = \eta.
$$
We denote  a unital Hom-associative algebra by a quadruple $(\A, \mu, \al, \eta)$.  The unit element (or unit, for simplicity) is $\mathds{1}= \eta(1_{\kk})$.
\end{defn}

Notice that a Hom-associative algebra $(\A, \mu, \al)$ is unital, with unit $\mathds{1} \in \A$,
when $ \mu(x, \mathds{1})= \mu(\mathds{1},x) =\al(x) $ and $ \al(\mathds{1})=\mathds{1}$.

Morphisms between Hom-associative algebras are defined in the similar way as Hom-Lie algebras. For unital Hom-associative algebras, the image of the unit is a unit.
\begin{ex}\cite{Yau09,MS}
\label{ex:composition}
Given a vector space $\gg$ equipped with a bilinear map $ \brr{\, , \, }:\gg \otimes \gg
\to \gg$ and an endomorphism $\al:\gg\to \gg$ of $(\gg, \brr{\, , \, })$. Define
$ \brr{\, , \, }_{\al}:\gg \otimes \gg
\to \gg$ by
$$ \brr{x,y}_\al=\al (\brr{x,y}), \quad \hbox{ $\forall x,y \in \gg$.} $$
 Then $(\gg, \brr{\, , \, }_{\al}, \al)$ is a Hom-Lie algebra (resp. a Hom-associative algebra, resp. unital Hom-associative algebra) if and only if the restriction of $\brr{\, , \, }$
to the image of $\al^2 $ is a Lie bracket (resp. an associative product, resp. a unital associative product).
In particular, Hom-Lie structures are naturally associated to Lie algebras equipped with a Lie algebra endomorphism \cite{Yau09}.
Such Hom-Lie structures are said to be \textbf{obtained by composition} or \textbf{Twisting principle}.
\end{ex}

\begin{ex} \label{ex:composition2}
As one can expect, the commutator of a Hom-associative algebra is a Hom-Lie algebra
\cite{MS}. More precisely, for every Hom-associative algebra $(\A, \mu, \al)$ (see Definition \ref{hom-associative} above), the triple $(\A, \brr{\, , \, }, \al)$  is a Hom-Lie algebra, where
$$
\brr{x,y}:= \mu(x,y)- \mu (y,x)
$$
for all $x,y \in \A$.
\end{ex}

\begin{defn}\label{def:our_inverse}
An element $x$ in a unital Hom-associative algebra $(\A, \mu, \al,\mathds{1}) $ is said to be \textbf{hom-invertible} if there exists an element $x^{-1}$ and a non-negative integer $k \in {\mathbb N}$, such that
\begin{equation}\label{eq:def_inverse}
\al^k \circ \mu(x, x^{-1})= \al^k \circ \mu(x^{-1}, x) =  \mathds{1}.
\end{equation}
 The element $x^{-1}$ is  called a \textbf{hom-inverse} and the smallest $k$ is the \textbf{invertibility index} of $x$.
\end{defn}
When it exists, the hom-inverse may not be unique, which prevents hom-invertible elements to be a Hom-group in the sense of Definition \ref{def:ourHomGroup}. However, the following can be shown.

\begin{prop}\label{prop:stableByProduct}
For every unital Hom-associative algebra $(\A, \mu, \al,\mathds{1}) $, the unit $\mathds{1}$ is  hom-invertible,
the product of any two hom-invertible elements is hom-invertible and every inverse of a hom-invertible element is hom-invertible.
 \end{prop}
\begin{proof}
 The only non-trivial point is that $\mu(x , x')$ is a hom-invertible element if both $x$ and $x'$ are, i.e.
 if there exists $y,y' \in \A$, $k,k' \in {\mathbb N}$ such that
    $$\al^k \circ \mu(x, y)= \al^k \circ \mu(y,x) = \al^{k'} \circ \mu(x', y')= \al^{k'} \circ \mu(y', x')  =  \mathds{1}.$$
 Hom-associativity implies
  $$ \al \circ \mu(\mu(x,x') , \mu(y',y))  = \mu( \al^2(x) , \mu(\mu(x',y'), \al(y) ) .
  $$

  So that
  \begin{eqnarray*} \al^{k+k'+1}  \circ \mu(\mu(x,x') , \mu(y',y)) & = & \al^k (   \mu( \al^{k'+2}(x) , \mu(\al^{k'} \mu(x',y'), \al^{k'+1}(y)  )
  \\ & =&  \al^k (   \mu( \al^{k'+2}(x) , \mu(\mathds{1}, \al^{k'+1}(y)  ) \\
    & =&  \al^k (   \mu( \al^{k'+2}(x) ,  \al^{k'+2}(y)  ) \\ &=&  \al^{k+k'+2} (   \mu( x , y  ) ) \\
    &=& \al^{k'+2}(\mathds{1}) = \mathds{1}.
  \end{eqnarray*}
 This completes the proof.
\end{proof}

\begin{rem}
For every unital Hom-associative algebra $(\A, \mu, \al,\mathds{1}) $, the subspace ${\mathfrak k}$ of all elements $x \in \A$ such that $ \al^k (x)=0$ for some integer $k \in {\mathbb N}$ is an ideal, i.e. the quotient map
$\A/{\mathfrak k}$ is  a unital Hom-associative algebra for which the induced map $\underline{\al}$ is invertible
 for the  induced product $ \underline{\mu}$.
In particular $ \A/{\mathfrak k}$ equipped with the product $\underline{\al}^{-1} \circ \underline{\mu} $ is an algebra.

An element $x\in \A$ is invertible in $(\A, \mu, \al,\mathds{1}) $ if and only if its image in $\A/{\mathfrak k}$ is invertible in the usual sense, which gives an alternative proof of Proposition \ref{prop:stableByProduct}.
\end{rem}

We now recall the notion of Hom-coalgebras.



\begin{defn} \cite{MS}
A \textbf{Hom-coalgebra} is a triple $(A, \Delta, \beta)$ where $A$ is a  vector space  and  $\Delta: A \to A\otimes A$, $\beta: A \to A$  are linear maps.

A \textbf{Hom-coassociative coalgebra} is a Hom-coalgebra $(A, \Delta, \beta)$ satisfying
$$
(\beta \otimes \Delta) \circ \Delta=(\Delta \otimes \beta) \circ \Delta.
$$
A Hom-coassociative coalgebra is said to be \textbf{co-unital}  if there exists a linear map $\epsilon: A \to \kk$ satisfying
$$
(id \otimes \epsilon) \circ \Delta= \beta \hbox{, }  (\epsilon \otimes id) \circ \Delta = \beta
 \hbox{ and }   \epsilon \circ \beta  = \epsilon.
$$
We refer to a counital Hom-coassociative coalgebra with a quadruple $(A,\Delta, \beta, \epsilon)$.
\end{defn}

Let $(A, \Delta, \beta)$ and $(A', \Delta', \beta')$ be two Hom-coalgebras (resp. Hom-coassociative algebras). A linear map $f: A \to A'$ is a morphism of Hom-coalgebras (resp. Hom-coassociative coalgebras) if
$$
(f \otimes f) \circ \Delta = \Delta' \circ f \qquad f \circ \beta = \beta' \circ f.
$$
It is said to be a weak Hom-coalgebras morphism if it holds only the first condition. If furthermore the Hom-coassociative coalgebras admit counits $\epsilon$ and $\epsilon'$, we have moreover $\epsilon = \epsilon' \circ f$.

The category of coassociative Hom-coalgebras is closed under weak Hom-coalgebra morphisms.

\begin{ex}
\cite{MS2009}
The dual of a Hom-algebra $(A,\mu,\al)$ is not always a Hom-coalgebra, because the coproduct does not land in the good space $\mu^*: A^* \to (A \otimes A)^* \supsetneq A^* \otimes A^*$.
Nevertheless, it is the case if the Hom-algebra is finite-dimensional, since $(A \otimes A)^*= A^*  \otimes A^*$.
The converse always  holds true. Let $(A, \Delta, \beta)$ be a Hom-coassociative coalgebra. Then its dual vector space is provided with a structure of Hom-associative algebra $(A^*, \Delta^*, \beta^*)$ where $\Delta^*$, $\beta^*$ are the transpose maps.
Moreover, the Hom-associative algebra is unital whenever $A$ is counital.

\cite{MS2010a} Let $(A,\Delta, \beta, \epsilon)$ be a counital Hom-coassociative coalgebra and $\al: A \to A$ be a weak Hom-coalgebra morphism. Then $(A,\Delta_{\al}= \Delta \circ \al, \beta \circ \al, \epsilon)$ is a counital Hom-coassociative coalgebra.
In particular, let $(A,\Delta,\epsilon)$ be a coalgebra and $\beta: A \to A$ be a coalgebra morphism. Then $(A,\Delta_{\beta}, \beta, \epsilon)$ is a counital Hom-coassociative coalgebra
\end{ex}

\begin{defn} \cite{MS2009}
An $(\al,\beta)$-\textbf{Hom-bialgebra} (simply called Hom-bialgebra when there is no ambiguity) is a heptuple $(A,\mu,\al, \eta,\Delta, \beta, \epsilon)$ where
\begin{enumerate}
  \item[(i)] $(A,\mu,\al, \eta)$ is a Hom-associative algebra with unit $\eta$ and unit element $\mathds{1}$,
  \item[(ii)] $(A,\Delta, \beta, \epsilon)$ is a Hom-coassociative coalgebra with a counit $\epsilon$,
  \item[(iii)] the linear maps $\Delta$ and $\epsilon$ are compatible with the multiplication $\mu$ and the unit $\eta$, that is for $x,y \in A$
\begin{enumerate}
  \item $\ds{\Delta(\mu(x \otimes y))= \Delta(x) \cdot \Delta(y)= \sum_{(x)(y)} \mu(x_1 \otimes y_1) \otimes \mu(x_2 \otimes y_2)}$, where $\cdot$ denotes the multiplication on the tensor algebra $A \otimes A$,
  \item $\Delta(\mathds{1})= \mathds{1} \otimes \mathds{1}$,
  \item $\epsilon(\mathds{1})= 1$,
  \item $\epsilon(\mu(x \otimes y))= \epsilon(x) \epsilon(y)$,
  \item $\epsilon \circ \al(x)= \epsilon(x)$.
\end{enumerate}
\end{enumerate}
If $\al=\beta$ the $(\al,\al)$-Hom-bialgebra is denoted by the hexuple $(A, \mu, \eta, \Delta, \epsilon, \al)$.
\end{defn}

A Hom-bialgebra morphism is a linear endomorphism which is simultaneously a Hom-algebra and Hom-coalgebra morphism.

\begin{ex}\label{TwistingPrincipleBialgebra}
Let $(A,\mu, \eta,\Delta,\epsilon, \al)$ be a Hom-bialgebra and $\beta:A \to A$ be a Hom-bialgebra morphism. Then $(A, \mu_{\beta} = \beta \circ \mu, \eta,\Delta_{\beta} = \Delta \circ \beta, \epsilon, \beta \circ \al)$ is a Hom-bialgebra.
In particular, if $(A,\mu, \eta,\Delta, \epsilon)$ is a bialgebra and $\beta: A \to A$ is a bialgebra morphism then $(A, \mu_{\beta}, \eta,\Delta_{\beta}, \epsilon, \beta )$ is a Hom-bialgebra.
This construction method of Hom-bialgebra, starting with a given Hom-bialgebra or a bialgebra and a morphism, is called composition method or  twisting principle \cite{MS2010a}. We can also define an $(\al, \beta)$ twist. If $(A,\mu, \eta,\Delta, \epsilon)$ is a bialgebra and $\al, \beta: A \to A$ are bialgebra morphisms which commute, that is $\al \circ \beta = \beta \circ \al$,
then $(A, \mu_{\al}=\al \circ \mu, \al,\eta,\Delta_{\beta}=\Delta \circ \beta, \beta,\epsilon )$ is a Hom-bialgebra. In particular we can consider one of the morphisms equal to identity.
\end{ex}

\section{Hom-Hopf algebras}

The following theorem holds and the proof goes through a direct verification of the axioms.

\begin{thm} \label{conv-product}
\cite{MS2009,MS2010a}
Let $(A,\mu,\al, \eta,\Delta, \beta, \epsilon)$ be an $(\al, \beta)$-Hom-bialgebra. Then $({\rm Hom}(A,A),\star,\gamma)$  is a unital Hom-associative algebra, with $\star$ being the multiplication given by the convolution product defined by
$$
f\star g = \mu \circ (f \otimes g) \circ \Delta
$$
and $ \gamma$ being the homomorphism of ${\rm Hom}(A,A)$ defined by $\gamma(f)= \al \circ f \circ \beta$. The unit is $\eta \circ \epsilon$.
\end{thm}

We now define Hom-Hopf algebra in a manner that differs from \cite{MS2009,MS2010a}. In those works, an antimorphism $S$ of $A$ is said to be an antipode if it is an inverse (in the usual sense) of the identity
over $A$ for the Hom-associative algebra ${\rm Hom} (A,A)$ with the multiplication given by the convolution product, i.e.
 $ S \star id =id \star S = \eta \circ  \epsilon $.
This definition matches examples given by twisting principle  out of a Hopf algebra but does not match our examples.
For this reason, and also in view of the Definition \ref{def:our_inverse} of an invertible element in the context of Hom-algebras, we prefer to define an antipode as being an anti-morphism $S$ of $A$ which is a relative hom-inverse of the identity, defined as follows. We say that $S \in {\rm Hom} (A,A)$ is a \textbf{relative hom-inverse} of $T \in {\rm Hom} (A,A)$ if and only if for every $x \in A$ there exists an integer $k$ (depending on $x$) such that:
$$ \al^{k} (S \star T ) \, (x) =  \al^k (T \star S ) \, (x)  = \eta \circ  \epsilon \, (x).$$

\begin{rem}
Notice that $S$ is not a hom-inverse of $T$ in the sense of Definition \ref{def:our_inverse}.
However,  if there exists an integer
$\bar{k}$ such that $ \al^{\bar{k}} (S \star T ) \, (x) =  \al^{\bar{k}} (T \star S ) \, (x)  = \eta \circ  \epsilon \, (x)$ for all $x\in A$, then   $S$ is a hom-inverse of $T$ and the smallest $\overline{k}$ is the invertibility index of $S$.\end{rem}


This amounts to the following definition:

\begin{defn} \label{def:hom-Hopf-algebra_in_our_sense}
Let $(A,\mu,\al, \eta,\Delta, \beta, \epsilon)$ be an $(\al, \beta)$-Hom-bialgebra.
An anti-homomorphism $S$ of $A$ is said to be an \textbf{antipode} if
\begin{enumerate}
\item[a)] $ S \circ \al = \al \circ S$,
\item[b)] $  S  \circ \eta= \eta $ and $  \epsilon  \circ S = \epsilon$,
\item[c)] $S$ is a relative Hom-inverse of the identity map $id :A \to A$ for the convolution product  given as in Theorem \ref{conv-product}.
\end{enumerate}
 An $(\al, \beta)$-\textbf{Hom-Hopf algebra} is an $(\al, \beta)$-Hom-bialgebra admitting an antipode.
\end{defn}

Recall that condition (c) means equivalently that, for every $x \in A $, there exists $ k \in {\mathbb N}$ such that:
 \begin{equation}\label{eq:antipode_condition}\al^k \circ( S \star id)  (x)=  \al^k \circ  (id \star S)  (x) = \eta \circ  \epsilon  (x).
  \end{equation}

Notice that we do not need to assume that $ S $ and $ \beta $ commute (in most examples,
$\beta$ is either the identity or coincides with $\al$).

\begin{ex} \label{Hom-Hopf-Twist}
Let  $(A,\mu, \eta,\Delta,  \epsilon, S)$ be a Hopf algebra, and $\al, \beta: A \to A$ be commuting bialgebra morphisms satisfying $S\circ \alpha=\alpha \circ S$.  Then
$$
(A, \mu_{\al}=\alpha\circ \mu,\al, \eta,\Delta_{\beta}=\Delta\circ\beta,  \beta,\epsilon ,S)
$$
is an $(\al, \beta)$-Hom-Hopf algebra, called the \textbf{$(\al,\beta)$-twist} of the Hopf algebra $A$.
More generally, the same idea turns a $(\al', \beta')$-Hom-Hopf algebra in a $(\al \circ \al', \beta \circ \beta')$-Hom-Hopf algebra.

Indeed, according to Example \ref{TwistingPrincipleBialgebra}, $
(A, \mu_{\al}=\alpha\circ \mu,\al, \eta,\Delta_{\beta}=\Delta\circ\beta,  \beta,\epsilon )
$ is a Hom-bialgebra. It remains to show that $S$ is still an antipode for the Hom-bialgebra. We have
$$S(\mu_\alpha(x,y))=S(\alpha(\mu(x,y)))=\alpha(S(\mu(x,y)))=\alpha(\mu(S(y),S(x)))=\mu_\alpha(S(y),S(x)),
$$
and
$$\mu_\alpha\circ(S\otimes \id)\circ\Delta_\beta=\alpha\circ\mu\circ(S\otimes \id)\circ\Delta\circ\beta=
\mu\circ(S\otimes \id)\circ\Delta\circ\alpha\circ\beta=\eta\circ\epsilon\circ\alpha\circ\beta=\eta\circ\epsilon,
$$
which complete the proof. The proof for the general case is similar.
\end{ex}

\begin{rem}
\label{rmk:SeveralPoints}
For every $(\al,\beta)$-Hom-Hopf algebra, the following properties hold:
\begin{enumerate}
 \item Using counitality, we have (in Sweedler's notation):
 $ \beta(x) = \sum x_1 \epsilon (x_2) = \sum \epsilon (x_1) \,  x_2$.
  \item Let $x$ be a primitive element (which means that $\Delta(x)= \mathds{1}\otimes x + x \otimes \mathds{1}$), then $\epsilon(x)= 0$.
  \item If $x$ and $y$ are two primitive elements in $A$, then we have $\epsilon(x)=0$ and the commutator $[x,y]= \mu(x \otimes y) - \mu (y \otimes x)$ is also a primitive element.
  \item The set of all primitive elements of $A$, denoted by ${\rm Prim}(A)$, admits a natural structure of Hom-Lie algebra, with bracket given by the commutator $ [x,y]:= \mu(x \otimes y) - \mu (y \otimes x)$, see \cite{MS2009,MS2010a}.
  \item If $x,y,z$ are  primitive elements in $A$, then the Hom-associator $\mu(\alpha(x),\mu(y,z))-\mu(\mu(x,y),\alpha(z))$ is a primitive element.
	 \item Using counitality and unitality, we have $ S \star (\eta \circ \epsilon) = \al \circ S \circ \beta =\gamma(S) $,
	and more generally $ (\al^p \circ S \circ \beta^q )\star (\eta \circ \epsilon) =  \al^{p+1} \circ S \circ \beta^{q+1} $.
	\item For all linear endomorphism $ S,T $ of $ A$, we have $\al ( S  \star T) = (\al \circ S) \star (\al \circ T) $.
	\item The antipode condition  (\ref{eq:antipode_condition}) can be stated as $ (\al^k \circ S) \star \al^k = \al^k \star (\al^k \circ S) = \eta \circ \epsilon$.
\end{enumerate}
\end{rem}

\begin{prop}
Let $(A,\mu,\al, \eta,\Delta, \beta, \epsilon) $ be an  $(\al,\beta)$-Hom-bialgebra. Assume that $S $ and $S'$ are two antipodes. Let $x\in A$ and  $k,k' \in {\mathbb N}$
  such that:
  $$  \al^{k} ( S \star \id )(x) =    \al^{k} ( \id \star S ) (x)=  \eta  \circ \epsilon(x) \hbox{ and }
	\al^{k'} ( S' \star \id )(x) =    \al^{k'} ( \id \star S' )(x) =  \eta  \circ \epsilon(x).$$
	Then, the following relation holds
 $$ \al^{K+2} \circ S \circ \beta^2(x) = \al^{K+2}  \circ S' \circ \beta^2 (x) $$
with $K ={\rm  max}(k,k')$.
\end{prop}
\begin{proof}We assume that $k'\leq k$.
Recall that  for any $f$, $f\star \eta\circ \epsilon =\alpha \circ f\circ \beta$. For simplicity, we omit the composition circle.
\begin{align*}
\alpha^{k+2}S'\beta^2&=\alpha(\alpha^{k+1}S'\beta)\beta=(\alpha^{k+1}S'\beta)\star (\eta \epsilon)=(\alpha^{k+1}S'\beta)\star(\alpha^k(\id \star S))=(\alpha^{k+1}S'\beta)\star(\alpha^k \star \alpha^k S).
\end{align*}
By Hom-associativity we have
\begin{align*}
\alpha^{k+2}S'\beta^2&=(\alpha^{k}S'\star\alpha^k)\star(\alpha^{k+1} S\beta)
=(\alpha^{k-k'}(\alpha^{k'}(S'\star\id))\star(\alpha^{k+1} S\beta)
=(\alpha^{k-k'}\eta\epsilon))\star(\alpha^{k+1} S\beta).
\end{align*}
Since $\alpha\eta=\eta$ and $\epsilon\beta=\epsilon$, we have
\begin{align*}
\alpha^{k+2}S'\beta^2&
=(\eta\epsilon)\star(\alpha^{k+1} S\beta)=(\alpha^{k+1}\eta\epsilon\beta)\star(\alpha^{k+1} S\beta)=\alpha^{k+1}(\eta\epsilon\star S)\beta=\alpha^{k+2}S\beta^2.
\end{align*}
\end{proof}
\begin{rem}
This proposition means  that the antipode is in some sense unique.
Indeed, when $\al$ and $\beta$ are invertible, the antipode is unique when it exists.
\end{rem}

\section{Elements of group-like type in an  $(\al,\beta)$-Hom-Hopf algebra}
\label{sec:grouplike}

For $( {A},\mu,\al,\mathds{1},\Delta,\beta, \epsilon, S) $ an $(\al,\beta)$-Hopf-algebra, all the structural maps
extend by ${\mathbb K}[[\nu]] $-linearity to yield an $(\al,\beta)$-Hom-bialgebra structure on the space ${A}[[\nu]]$ of formal series with coefficients in $A$. However, it may not be an $(\al,\beta)$-Hom-Hopf algebra because
$S$ may not be a relative-inverse of the identity map.
However, for  formal series $g(\nu)=\sum_{i \geq 0} g_i \nu^i$  such that  the invertibility indexes of the elements $(g_i)_{i \in {\mathbb N}}$ are  bounded, there exits $k\in {\mathbb N}$ such that
$$ \al^k\circ(S\star \id )(g_i)= \al^k\circ(\id \star S )(g_i)=\eta \circ \epsilon (g_i).$$
Summing up these  relations, we obtain:

\begin{prop}
The space $A_{b}[[\nu]]$ of formal series $g(\nu)=\sum_{i \geq 0} g_i \nu^i$, such that  the invertibility indexes of the elements  $(g_i)_{i \in {\mathbb N}}$  are bounded,
is  an $(\al,\beta)$-Hom-Hopf algebra.
\end{prop}
In particular,  polynomial elements are in $ A_{b}[[\nu]] $.

\begin{defn}
\label{def:groupLikeAndTheLikes}
Let $( {\A},\mu,\al,\mathds{1},\Delta,\beta , \epsilon, S) $ be an $(\al,\beta )$-Hopf-algebra.
An element $g\in \A$ is a {\textbf{ group-like element}} if
$ \Delta (g) = g \otimes g \hbox{ and } \epsilon(g)=1$.
Let  $\nu $ be a formal parameter.
\begin{enumerate}
 \item[(i)] We call a {\textbf{formal group-like element}} a  formal series $g(\nu) \in {\A}_b[[\nu]] $ such that:
\begin{equation}\label{eq:group-like} \Delta (g(\nu)) = g(\nu) \otimes g(\nu) \hbox{ and } \epsilon(g(\nu))=1 .\end{equation}
 \item[(ii)]
Elements in ${A}[\nu] $ (i.e. polynomials in $\nu$) are called {\textbf{$p$-order group-like elements}}
when:
\begin{equation}\label{eq:group-like-order-n} \Delta (g(\nu)) = g(\nu) \otimes g(\nu)   \hbox{ modulo }  \nu^{p+1} \hbox{ and } \epsilon(g(\nu))=1.\end{equation}
 \item [(iii)]
A {\textbf{formal group-like sequence}} is a sequence $ (g_p (\nu))_{p \in {\mathbb N}} $ of elements in ${\A}[\nu] $ such that
 \begin{enumerate}
 \item[a)] for all $p \geq 1$, $g_p(\nu) $ is a $p$-order formal group-like element,
 \item[b)]  $g_{p+1} (\nu)= \al (g_p (\nu))$ modulo $\nu^{p+1} $,
 \item[c)] there exists an integer $k$ such that the invertibility index of  $g_p(\nu) $
is less or equal to  $k$ for all $p \in {\mathbb N}$.
 \end{enumerate}
 We denote by $G_{seq}(A)$ the set of all formal group-like  sequences of ${\A} $.
\end{enumerate}
\end{defn}

It is classical that group-like elements form a group.
More generally:

\begin{prop}
\label{prop:grouplike}
Let $( {A},\mu,\al,\mathds{1},\Delta,\beta , \epsilon, S) $  be an $(\alpha,\beta)$-Hom-Hopf algebra.
Group-like elements, formal group-like elements, $p$-order group-like elements for an arbitrary $p \in {\mathbb N}$, and formal group-like sequences form a Hom-group. Its product is $\mu$, the inverse of $g$ is $S(g)$, and the unit is ${\mathds{1}}$.
\end{prop}
\begin{proof}
 Stability under $\mu$ of group-like elements, $k$-order group-like elements and formal group-like sequences is an immediate consequence of the compatibility relation  between the multiplication and the comultiplication.
 The fact that group-like elements are hom-invertible and that $S(g)$ is a hom-inverse of a group-like element $g$ follows from \eqref{eq:antipode_condition}, which implies that there exists $k \in { \mathbb N}$ such that
  $$ \al^k \circ \mu  \,( S \otimes \id)\circ  \Delta (g) =   \al^k \circ \mu  \, ( \id \otimes S)\circ  \Delta (g) = \eta \circ \epsilon \, (g) ,$$
	which amounts the relation
	$$ \al^k \circ \mu \, (S(g),   g ) =  \al^k \circ \mu \, (g ,  S(g) )  = {\mathds{1}} .$$
 The same applies to $p$-order group-like elements, by taking the previous relation modulo $\nu^{p+1}$.
It then applies for formal group-like sequences, upon noticing that the assumption
(iii), c) in Definition \ref{def:groupLikeAndTheLikes}
guaranties that $ S(g_p (\nu)) $, which is a Hom-inverse of $g_p(\nu)$, has
an invertibility index bounded independently from $p$, hence that $ (S(g_p (\nu)))_{p \in {\mathbb N}} $ is a Hom-inverse of $ (g_p (\nu))_{p \in {\mathbb N}} $.
 For formal group-like elements, the proof follows the same lines, upon noticing that for every formal group-like element $g(\nu)=\sum_{i= 0}^\infty g_i \nu^i$, there is by assumption an integer $k$ such that for all $i \in {\mathbb N}$:
$$ \al^k \circ \mu ( S \otimes \id)\circ  \Delta (g_i) =   \al^k \circ \mu ( \id \otimes S)\circ  \Delta (g_i) = \eta \circ \epsilon \, (g_i) = \mathds{1} ,$$
so that $S(g(\nu))=\sum_{i= 0}^\infty S(g_i) \nu^i$ is a Hom-inverse of $g(\nu)$.
\end{proof}

\begin{rem}\label{rem:functorgroup}
Also, any morphism $\Phi$ of  $(\alpha,\beta)$-Hom-Hopf algebra induces in an obvious manner a morphism $\underline{\Phi}$ of Hom-groups between their respective group-like elements, formal group-like elements, $p$-order group-like elements for an arbitrary $p \in {\mathbb N}$, and formal group-like sequences.
Assigning to  an $(\alpha,\beta)$-Hom-Hopf algebra any of the previous types of Hom-groups, one obtains therefore a functor.
We call $G_{seq}$ the functor which associates to an $(\alpha,id)$-Hom-Hopf algebra its Hom-group of formal group-like sequences.
\end{rem}
\section{Weighted trees and universal enveloping algebras as Hom-Hopf algebras}
\label{univ_envel_algebra}

Donald Yau \cite{Yau08,Yau4} associated to any Hom-Lie algebra $(\gg, \brr{\, , \, }, \al)$ a Hom-associative algebra $ (U\gg,\mu,\al_F)$, that he called the universal enveloping algebra of $\gg$ and proved to be a Hom-bialgebra. His construction went through two steps: he first associated to any vector space $E$, equipped with a linear map $\al: E \to E$, the \textbf{free Hom-nonassociative algebra} $(F_{HNAs}(E), \mu_F, \al_F)$,
then, he considered the quotient of this algebra through the ideal $I^{\infty}=\bigcup_{n \geq 1} I^n$ where $I^1$ is the two-sided ideal
$$
I^1= \langle {\rm Im}(\mu_F \circ (\mu_F \otimes \al_F - \al_F \otimes \mu_F)); [x,y] -(xy-yx) \mbox{ for } x,y \in \gg \rangle, $$
and $(I^n)_{n \in {\mathbb N}} $ is given by the recursion relation $ I^{n+1}= \langle I^n \cup \al(I^n)\rangle $. He did not show that the henceforth obtained Hom-bialgebra comes with an antipode.

In the multiplicative case, a more direct construction exists, that we now present. We need first to introduce some generalities about weighted trees, and to define the free Hom-associative multiplicative algebra.
Moreover, we will see that it carries a structure of Hom-Hopf algebra.

\subsection{Weighted trees as the free Hom-associative algebra with $1$-generator}

A planar tree is an oriented graph drawn on a plane  with only one root. It is called binary when any vertex is trivalent, i.e., one root and two leaves. Usually we draw the root at the bottom of the tree and the leaves are drawn at the top of it:
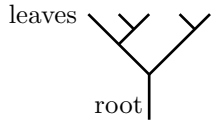
\begin{center}
\begin{tikzpicture}[xscale=0.2, yscale=0.2]

\draw[line width=1pt] (0,-1) -- (0,2) -- (4,6);
\draw[line width=1pt] (3,5) -- (2,6);
\draw[line width=1pt] (0,2) -- (-4,6);
\draw[line width=1pt] (-2,4) -- (0,6);
\draw[line width=1pt] (-1,5) -- (-2,6);

\draw (-2,0) node {root};
\draw (-7,6) node {leaves};

\end{tikzpicture}
\end{center}

For any natural number $n\geq 1$, let $T_n$ denote the set of planar binary trees with $n$ leaves and one root. For $n=1$, $T_1$ admits only one element, namely the unique tree with one leaf and the root. Below are the sets $T_n$ for $n=1,2,3,4$:
$$
T_1=\left\{ \mbox{\begin{tikzpicture}[xscale=0.4, yscale=0.4,baseline={([yshift=-.8ex]current bounding box.center)}]
\draw[line width=1pt] (0,0) -- (0,2);
\end{tikzpicture}} \: \right\}, \quad T_2= \left\{ \mbox{
\begin{tikzpicture}[xscale=0.4, yscale=0.4,baseline={([yshift=-.8ex]current bounding box.center)}]
\draw[line width=1pt] (0,0) -- (0,1) -- (1,2);
\draw[line width=1pt] (0,1) -- (-1,2);
\end{tikzpicture}
} \right\} , \quad T_3=\left\{ \mbox{
\begin{tikzpicture}[xscale=0.4, yscale=0.4,baseline={([yshift=-.8ex]current bounding box.center)}]
\draw[line width=1pt] (0,0) -- (0,1) -- (1,2);
\draw[line width=1pt] (0.5,1.5) -- (0,2);
\draw[line width=1pt] (0,1) -- (-1,2);
\end{tikzpicture}
}, \mbox{
\begin{tikzpicture}[xscale=0.4, yscale=0.4,baseline={([yshift=-.8ex]current bounding box.center)}]
\draw[line width=1pt] (0,0) -- (0,1) -- (1,2);
\draw[line width=1pt] (0,1) -- (-1,2);
\draw[line width=1pt] (-0.5,1.5) -- (0,2);
\end{tikzpicture}
} \right\} $$
$$ T_4=\left\{ \mbox{
\begin{tikzpicture}[xscale=0.4, yscale=0.4,baseline={([yshift=-.8ex]current bounding box.center)}]
\draw[line width=1pt] (0,0) -- (0,1) -- (1,2);
\draw[line width=1pt] (0.3,1.3) -- (-0.2,2);
\draw[line width=1pt] (0.6,1.6) -- (0.35,2);
\draw[line width=1pt] (0,1) -- (-1,2);
\end{tikzpicture}
}, \mbox{
\begin{tikzpicture}[xscale=0.4, yscale=0.4,baseline={([yshift=-.8ex]current bounding box.center)}]
\draw[line width=1pt] (0,0) -- (0,1) -- (1,2);
\draw[line width=1pt] (0.3,1.3) -- (-0.2,2);
\draw[line width=1pt] (0.07,1.6) -- (0.45,2);
\draw[line width=1pt] (0,1) -- (-1,2);
\end{tikzpicture}
},\mbox{
\begin{tikzpicture}[xscale=0.4, yscale=0.4,baseline={([yshift=-.8ex]current bounding box.center)}]
\draw[line width=1pt] (0,0) -- (0,1) -- (1,2);
\draw[line width=1pt] (0,1) -- (-1,2);
\draw[line width=1pt] (0.6,1.6) -- (0.2,2);
\draw[line width=1pt] (-0.6,1.6) -- (-0.2,2);
\end{tikzpicture}
}, \mbox{
\begin{tikzpicture}[xscale=0.4, yscale=0.4,baseline={([yshift=-.8ex]current bounding box.center)}]
\draw[line width=1pt] (0,0) -- (0,1) -- (1,2);
\draw[line width=1pt] (-0.3,1.3) -- (0.2,2);
\draw[line width=1pt] (-0.6,1.6) -- (-0.35,2);
\draw[line width=1pt] (0,1) -- (-1,2);
\end{tikzpicture}
}, \mbox{
\begin{tikzpicture}[xscale=0.4, yscale=0.4,baseline={([yshift=-.8ex]current bounding box.center)}]
\draw[line width=1pt] (0,0) -- (0,1) -- (1,2);
\draw[line width=1pt] (-0.3,1.3) -- (0.2,2);
\draw[line width=1pt] (-0.07,1.6) -- (-0.45,2);
\draw[line width=1pt] (0,1) -- (-1,2);
\end{tikzpicture}
}\right\}.
$$
An element  $\varphi \in T_n$ shall be called an $n$-tree for short. When necessary we label the leaves of an $n$-tree by $1,2,3, \dots, n$ from left to right.

For $\varphi \in T_n$ and $\psi \in T_m$ be a pair of trees, the $(n+m)$-tree $\varphi  \vee \psi$, called the \emph{grafting of $\varphi$ and $\psi$}, is obtained by joining the roots of $\varphi $ and $\psi$ to create a new root. For instance,
\begin{center}
\begin{tikzpicture}[xscale=0.4, yscale=0.4]
\draw[line width=1pt] (0,0) -- (0,1) -- (1,2);
\draw[line width=1pt] (0,1) -- (-1,2);
\draw[line width=1pt] (-0.5,1.5) -- (0,2);
\draw (2,1) node {$\vee$};
\end{tikzpicture}
 \begin{tikzpicture}[xscale=0.4, yscale=0.4]
\draw[line width=1pt] (0,0) -- (0,1) -- (1,2);
\draw[line width=1pt] (0,1) -- (-1,2);
\draw (2,1) node {$=$};
\end{tikzpicture}
\begin{tikzpicture}[xscale=0.3, yscale=0.3]
\draw[line width=1pt] (1,0) -- (0,1) -- (1,2);
\draw[line width=1pt] (0,1) -- (-1,2);
\draw[line width=1pt] (-0.5,1.5) -- (0,2);
\draw[line width=1pt] (1,0) -- (1,-1);
\draw[line width=1pt] (1,0) -- (3,2);
\draw[line width=1pt] (2.5,1.5) -- (2,2);
\end{tikzpicture}
\end{center}
Note that grafting is neither an associative nor a commutative operation. For any tree $\varphi \in T_n$, there are unique integers $p$ and $q$ with $p+q=n$ and trees $\varphi_1 \in T_p$ and $\varphi_2 \in T_q$ such that $\varphi= \varphi_1 \vee \varphi_2$. It is clear that any tree in $T_n$ can be obtained from \begin{tikzpicture}[xscale=0.4, yscale=0.4,baseline={([yshift=-.8ex]current bounding box.center)}]
\draw[line width=1pt] (0,0) -- (0,1);
\end{tikzpicture} , the 1-tree,  by sucessive graftings.

 Yau's construction of $U\gg $ used weighted trees. In the sequel, since we are dealing with multiplicative case, it suffices to work with leaf weighted trees:

\begin{defn}
A \textbf{leaf weighted $n$-tree} is a pair $(\varphi, a)$ where:
\begin{itemize}
\item $\varphi \in T_n$ is a $n$-tree,
\item $a$ is an $n$-tuple $(a_1, a_2, \dots,a_n) \in {\mathbb N}^n$ of non-negative integers.
\end{itemize}
We call the tree $\varphi$ the underlying tree of the leaf weighted $n$-tree $(\varphi,a)$ while,
for all $ i=1 \dots,n$, the integer $a_i$ shall be referred to as the \textbf{weight} of the leaf $i$.
\end{defn}


We will indeed barely use the notation $ (\varphi,a) $ at all, and find more convenient to picture a leaf weighted $n$-tree $(\varphi, a_1,a_2, \dots, a_n)$ by drawing the tree $\varphi$ and putting the weight $a_i$ next to each leaf. For example, here are two leaf weighted $3$-trees:
\begin{center}
\begin{tikzpicture}[xscale=0.25, yscale=0.25]
\draw[line width=1pt] (0,0) -- (0,2) -- (2,4);
\draw[line width=1pt] (0,2) -- (-2,4);
\draw[line width=1pt] (-1,3) -- (0,4);


\draw (-2,4.7) node {$0$};
\draw (0,4.7) node {$2$};
\draw (2,4.7) node {$1$};
\end{tikzpicture}\hspace*{2cm}
\begin{tikzpicture}[xscale=0.25, yscale=0.25]

\draw[line width=1pt] (0,0) -- (0,2) -- (-2,4);
\draw[line width=1pt] (0,2) -- (2,4);
\draw[line width=1pt] (1,3) -- (0,4);
\draw (-2,4.7) node {$0$};
\draw (0,4.7) node {$1$};
\draw (2,4.7) node {$0$};
\end{tikzpicture}
\end{center}

The grafting operation extends to leaf weighted $n$-trees. For example:
\begin{center}
\begin{tikzpicture}[xscale=0.25, yscale=0.25]
\draw[line width=1pt] (0,0) -- (0,2) -- (2,4);
\draw[line width=1pt] (0,2) -- (-2,4);
\draw[line width=1pt] (-1,3) -- (0,4);

\draw (-2,4.7) node {$0$};
\draw (0,4.7) node {$2$};
\draw (2,4.7) node {$1$};
\draw (3,2) node {$\vee$};
\end{tikzpicture}
\begin{tikzpicture}[xscale=0.25, yscale=0.25]
\draw[line width=1pt] (0,0) -- (0,2) -- (-2,4);
\draw[line width=1pt] (0,2) -- (2,4);
\draw[line width=1pt] (1,3) -- (0,4);

\draw (-2,4.7) node {$0$};
\draw (0,4.7) node {$1$};
\draw (2,4.7) node {$0$};
\draw (3,2) node {$=$};
\end{tikzpicture}
\begin{tikzpicture}[xscale=0.4, yscale=0.4]
\draw[line width=1pt] (0,1) -- (0,1.5) -- (-1.5,3) -- (-2.5,4);
\draw[line width=1pt] (-2,3.5) -- (-1.5,4);
\draw[line width=1pt] (-1.5,3) -- (-0.5,4);
\draw[line width=1pt] (0,1.5) -- (1.5,3) -- (2.5,4);
\draw[line width=1pt] (1.5,3) -- (0.5,4);
\draw[line width=1pt] (2,3.5) -- (1.5,4);


\draw (-2.5,4.4) node {$0$};
\draw (-1.5,4.4) node {$2$};
\draw (-0.5,4.4) node {$1$};
\draw (0.5,4.4) node {$0$};
\draw (1.5,4.4) node {$1$};
\draw (2.5,4.4) node {$0$};
\end{tikzpicture}
\end{center}

For all $n \geq 1$, we let $B_n$ denote the set of leaf weighted $n$-trees. Let $B$ denote the union over $n \in {\mathbb N}$ of the sets $B_n$ together with an element that we call the unit and denote by $\mathds{1}$. Note that the element $\mathds{1}$ is different from the leaf weighted 1-tree $\begin{tikzpicture}[baseline={([yshift=-.8ex]current bounding box.center)}]
\draw[line width=1pt] (0,0) -- (0,0.4);
\draw (0,0.6) node {$0$};
\end{tikzpicture}$. We then consider the free vector space ${\mathds{T}}$ generated by the set~$B$.

We define on ${\mathds{T}}$ two natural linear maps:
\begin{enumerate}
 \item[(i)]  $\al:{\mathds{T}} \to {\mathds{T}} $ sending  $\mathds{1}$ to $\mathds{1}$, i.e. $\al(\mathds{1})=\mathds{1}$  and sending a leaf weighted $n$-tree to the leaf weighted $n$-tree obtained  adding $+1$ to all the weights of the leaves, i.e. $\al((\varphi, a_1,a_2, \dots, a_n))=((\varphi, a_1+1,a_2+1, \dots, a_n+1))$;
 \item[(ii)] a product $ \vee $ that, for  any pair of leaf weighted trees, is just the grafting of these trees and such that for any weighted tree $\varphi $
 $$
 \varphi \vee \mathds{1} = \mathds{1} \vee \varphi = \al (\varphi)
 $$
and $\mathds{1}\vee \mathds{1} = \mathds{1}$.
\end{enumerate}

For example,
$$\al\left( \mbox{\begin{tikzpicture}[xscale=0.2, yscale=0.2,baseline={([yshift=-.8ex]current bounding box.center)}]

\draw[line width=1pt] (0,0) -- (0,2) -- (2,4);
\draw[line width=1pt] (0,2) -- (-2,4);
\draw[line width=1pt] (-1,3) -- (0,4);


\draw (-2,4.8) node {$0$};
\draw (0,4.8) node {$2$};
\draw (2,4.8) node {$1$};
\end{tikzpicture}} \right)= \mbox{ \begin{tikzpicture}[xscale=0.2, yscale=0.2,baseline={([yshift=-.8ex]current bounding box.center)}]

\draw[line width=1pt] (0,0) -- (0,2) -- (2,4);
\draw[line width=1pt] (0,2) -- (-2,4);
\draw[line width=1pt] (-1,3) -- (0,4);

\draw (-2,4.8) node {$1$};
\draw (0,4.8) node {$3$};
\draw (2,4.8) node {$2$};
\end{tikzpicture}}
$$

The proof of  the following lemma is trivial:

\begin{lem} \label{alphavee}
The map $\al$ defined in item (i) above is a morphism for the grafting of trees, that is
$$
\al(\varphi \vee \psi)= \al(\varphi)\vee \al(\psi)
$$
for any $\varphi, \psi \in {\mathds{T}}$.
\end{lem}



We now define an important operation which consists in eliminating leaves while changing weights of the remaining ones:

\begin{defn}
Let $\varphi \in B_n$ and $I \subset \{ 1,2, \dots, n \}$, we will denote by $\varphi_I$ the tree obtained by replacing
all the leaves in $\{ 1,2, \dots, n \} \backslash I$ by $\mathds{1}$. In particular, $\varphi_{\emptyset}=\mathds{1}$ and $\varphi_{\{ 1,2, \dots, n \}}= \varphi$.
\end{defn}

As an example, if $\varphi$ is the tree \begin{tikzpicture}[xscale=0.4, yscale=0.4,baseline={([yshift=-.8ex]current bounding box.center)}]
\draw[line width=1pt] (0,1) -- (0,1.5) -- (-1.5,3) -- (-2.5,4);
\draw[line width=1pt] (-2,3.5) -- (-1.5,4);
\draw[line width=1pt] (-1.5,3) -- (-0.5,4);
\draw[line width=1pt] (0,1.5) -- (1.5,3) -- (2.5,4);
\draw[line width=1pt] (1.5,3) -- (0.5,4);
\draw[line width=1pt] (2,3.5) -- (1.5,4);


\draw (-2.5,4.5) node {$2$};
\draw (-1.5,4.5) node {$4$};
\draw (-0.5,4.5) node {$0$};
\draw (0.5,4.5) node {$3$};
\draw (1.5,4.5) node {$1$};
\draw (2.5,4.5) node {$2$};
\end{tikzpicture} and $I=\{3,5,6\}$, then

$$\varphi_I = \begin{tikzpicture}[xscale=0.4, yscale=0.4,baseline={([yshift=-.8ex]current bounding box.center)}]

\draw[line width=1pt] (0,1) -- (0,1.5) -- (-1.5,3) -- (-2.5,4);
\draw[line width=1pt] (-2,3.5) -- (-1.5,4);
\draw[line width=1pt] (-1.5,3) -- (-0.5,4);
\draw[line width=1pt] (0,1.5) -- (1.5,3) -- (2.5,4);
\draw[line width=1pt] (1.5,3) -- (0.5,4);
\draw[line width=1pt] (2,3.5) -- (1.5,4);


\draw (-2.5,4.5) node {$\mathds{1}$};
\draw (-1.5,4.5) node {$\mathds{1}$};
\draw (-0.5,4.5) node {$0$};
\draw (0.5,4.5) node {$\mathds{1}$};
\draw (1.5,4.5) node {$1$};
\draw (2.5,4.5) node {$2$};
\end{tikzpicture} =
\begin{tikzpicture}[xscale=0.8, yscale=0.8,baseline={([yshift=-.8ex]current bounding box.center)}]
\draw[line width=1pt] (0,0.7) -- (0,1) -- (1,2);
\draw[line width=1pt] (0,1) -- (-1,2);
\draw[line width=1pt] (0.6,1.6) -- (0.2,2);
\draw[line width=1pt] (-0.6,1.6) -- (-0.2,2);
\draw (-1,2.2) node {$\mathds{1}$};
\draw (-0.2,2.25) node {$0$};
\draw (0.2,2.25) node {$2$};
\draw (1,2.25) node {$3$};
\end{tikzpicture} =
\begin{tikzpicture}[xscale=0.8, yscale=0.8,baseline={([yshift=-.8ex]current bounding box.center)}]
\draw[line width=1pt] (0,0.7) -- (0,1) -- (1,2);
\draw[line width=1pt] (0,1) -- (-1,2);
\draw[line width=1pt] (0.6,1.6) -- (0.2,2);
\draw (-1,2.2) node {$1$};
\draw (0.2,2.25) node {$2$};
\draw (1,2.25) node {$3$};
\end{tikzpicture}
$$

\

We then define a coproduct $\Delta: {\mathds{T}}  \longrightarrow {\mathds{T}}  \otimes {\mathds{T}} $ by
\begin{equation}\label{def:coproduct}
\Delta \varphi= \sum_{\substack{I\cup J= \{1 , \dots, n\} \\ I \cap J= \emptyset}}  \varphi_I \otimes \varphi_J
\end{equation}
for a leaf weighted $n$-tree $\varphi$, extended by linearity, and $\Delta (\mathds{1}) = \mathds{1} \otimes \mathds{1} $.

\begin{lem}\label{lem:coproduct-properties}
This coproduct satisfies the following:
\begin{enumerate}
  \item[(i)] $(\Delta \otimes id) \circ \Delta = ( id \otimes \Delta) \circ \Delta$, i.e., $\Delta$ is coassociative.
  \item[(ii)] $\sigma_{12} \circ \Delta = \Delta$, i.e., $\Delta$ is cocommutative (with $\sigma_{12}: A \otimes B \longrightarrow B \otimes A$ being the twist map defined by $\sigma_{12}(a \otimes b)= b \otimes a$.)
\end{enumerate}
\end{lem}

\begin{proof}
(i) Using Definition \ref{def:coproduct} of $\Delta$, we have
$$
(\Delta \otimes id) \circ \Delta \varphi= \sum_{\substack{I \cup J \cup K= \{1,\dots, n\} \\ I \cap J= \emptyset \\ I \cap K = \emptyset \\ J \cap K = \emptyset}} \varphi_I \otimes \varphi_J \otimes \varphi_K =( id \otimes \Delta) \circ \Delta \varphi
$$
for any $\varphi \in B_n$.

(ii) Trivial in view of (\ref{def:coproduct}).
\end{proof}

We also define a counit $\epsilon: {\mathds{T}} \longrightarrow \kk$ by $\epsilon(\mathds{1})= 1$ and $\epsilon( \varphi)=0$, for any $\varphi \in B_n$, using then linearity. Using the above Lemma \ref{lem:coproduct-properties}, it is easy to prove that:

\begin{prop}\label{prop:coalgebra}
The triple $({\mathds{T}}, \Delta,\epsilon)$ is a co-unital coassociative cocommutative coalgebra.
\end{prop}

This co-unital coassociative cocommutative coalgebra is compatible with the product $\vee$ in the following sense:

\begin{lem} \label{Deltavee}
For all  $\varphi, \psi \in {\mathds{T}} $  the compatibility relation
\begin{equation}\label{compatibility}
\Delta( \varphi \vee \psi)= \Delta(\varphi) \vee \Delta(\psi),
\end{equation}
holds with the understanding that the right hand side of (\ref{compatibility}) is equipped with the natural algebra structure on the tensor algebra.
%
\end{lem}

\begin{proof}
For $I$   a subset of $\{1, \dots, n\}$, let us denote by $I^c$ the complement set.
For all $\varphi \in B_n$ and $\psi \in B_m$, equation (\ref{def:coproduct}) gives
$$
\Delta( \varphi \vee \psi) = \sum_{I \subseteq \{1, \dots, n+m\}} (\varphi \vee \psi)_I \otimes (\varphi \vee \psi)_{I^c}.
$$
Let $I_n= I \cap \{1,\dots,n\}$ and $I_m = I \cap \{n+1, \dots, n+m\}$, and define $I_n^c$ and $I_m^c$ to be the complements of $I_n $ and $ I_m$ in $\{1,\dots,n\}$
and $ \{n+1, \dots, n+m\}$, respectively. A direct computation gives:
\begin{align}
 \Delta( \varphi \vee \psi)& = \sum_{\substack{I_n \subseteq \{1, \dots, n\} \\ I_m \subseteq \{n+1, \dots, n+m\}}} (\varphi_{I_n} \vee \psi_{I_m}) \otimes (\varphi_{I_n^c} \vee \psi_{I_m^c}) \nonumber \\
& = \sum_{\substack{I_n \subseteq \{1, \dots, n\} \\ I_m \subseteq \{1, \dots, m\}}} (\varphi_{I_n} \otimes \varphi_{I_n^c}) \vee (\psi_{I_m} \otimes \psi_{I_m^c}) \nonumber \\
& = \left(\sum_{I_n \subseteq \{1, \dots, n\}} \varphi_{I_n} \otimes \varphi_{I_n^c}\right)   \vee \left(\sum_{ I_m \subseteq \{1, \dots, m\}}\psi_{I_m} \otimes \psi_{I_m^c} \right) \nonumber \\
& = \Delta(\varphi) \vee \Delta(\psi). \nonumber
\end{align}
If $ \varphi$ or $\psi $ are equal to $\mathds{1}$, (\ref{compatibility}) is equivalent to:
\begin{equation}
\label{Deltaal}
\Delta \circ \al= (\al \otimes \al) \circ \Delta,
\end{equation}
which follows directly from (\ref{def:coproduct}). This completes the proof.
\end{proof}

Because the map $\vee$ is not associative the set $({\mathds{T}}, \vee, \mathds{1}, \Delta, \epsilon)$ is not a bialgebra. Let us consider now the bilateral ideal ${\mathcal I}$, with respect to $\vee$, generated by the following elements:
 $$
 (\phi \vee \psi) \vee \al (\chi) -\al (\phi) \vee (\psi \vee \chi)
 $$
where $\phi,\psi,\chi $  are either arbitrary leaf weighted  trees or the unit $\mathds{1}$.
Notice that in case $\phi, \psi$ or $\chi$ are equal to $\mathds{1}$, then $c(\phi,\psi,\chi)=0$, using Lemma \ref{alphavee}.
We call the quotient $\mathds{T}/{\mathcal I}$,  equipped with its structural maps $\vee, \al$ and the unit $\mathds{1}$,
the \textbf{free Hom-associative algebra with $1$-generator}.

\begin{rem}\label{rem:universal}
The free Hom-associative algebra with $1$-generator should be considered as an algebra that encodes all the possible operations
that can be described purely in terms of the Hom-associative product and $\al$. These operations can then
be applied to an arbitrary Hom-associative algebra $\A$.
\end{rem}

\begin{prop}
The coproduct $\Delta$ and the counit $\epsilon$ of $\mathds{T}$ go to the quotient, and endow
the free Hom-associative algebra with $1$-generator $\mathds{T}/{\mathcal I}$ with a Hom-bialgebra
structure.
\end{prop}

We first prove an obvious lemma:

\begin{lem}\label{lem:coideal1}
Let $e,f,g,h$ be elements in $ {\mathds{T}}$.
If $e - f \in {\mathcal I}$ and $g-h \in {\mathcal I}$, then $e \otimes g - f \otimes h \in {\mathcal I} \otimes {\mathds{T}} + {\mathds{T}} \otimes {\mathcal I}$.
\end{lem}
\begin{proof}
We just have to use the algebraic identity $e \otimes g - f \otimes h= (e-f) \otimes g + f \otimes (g-h)$.
\end{proof}

\begin{lem}\label{lem:coideal2}
The ideal ${\mathcal I}$ is a coideal of $({\mathds{T}}, \Delta,\epsilon)$.
\end{lem}
\begin{proof}
We have to check that:
\begin{itemize}
  \item $\epsilon({\mathcal I})=0$.
  \item  $\Delta$ maps ${\mathcal I} $ to ${\mathcal I} \otimes {\mathds{T}} + {\mathds{T}} \otimes {\mathcal I}$.
 \end{itemize}
 The first point is trivial by definition of $\epsilon$.
 For the second point,
since $\Delta$ and $\vee$ are compatible (Lemma \ref{Deltavee}) it suffices indeed to show that generating elements of ${\mathcal I} $, i.e.
elements of the form $c(\phi,\psi,\chi)=( (\phi \vee \psi) \vee \al (\chi) -\al (\phi) \vee (\psi \vee \chi) $,
 are mapped to ${\mathcal I} \otimes {\mathds{T}} + {\mathds{T}} \otimes {\mathcal I} $.

For arbitrary $\phi \in B_n$, $\psi \in B_m$ and $\chi \in B_k$, we have on the one hand:
\begin{align}
  \Delta\left( (\phi \vee \psi) \vee \al (\chi) \right)
 & = \left(\Delta(\phi) \vee \Delta(\psi) \right) \vee \Delta\left( \al (\chi) \right), \quad \mbox{by Lemma  \ref{Deltavee}} \nonumber \\
 &= \left(\Delta(\phi) \vee \Delta(\psi) \right) \vee \al ^{\otimes^2}\left( \Delta (\chi) \right), \quad \mbox{by equation (\ref{Deltaal})} \nonumber \\
& = \sum_{\substack{I \subseteq \{1,\dots, n\}  \\ J \subseteq \{1,\dots, m\} \\K \subseteq \{1,\dots, k\}}}  \left( \left( \phi_I \otimes \phi_{I^c} \right) \vee \left(\psi_J \otimes \psi_{J^c} \right) \right) \vee \left( \al (\chi_K) \otimes \al (\chi_{K^c})\right) \nonumber \\
&  = \sum_{\substack{I \subseteq \{1,\dots, n\}  \\ J \subseteq \{1,\dots, m\} \\K \subseteq \{1,\dots, k\}}} \left( \left( \phi_I \vee \psi_J \right) \vee  \al (\chi_K) \right) \otimes \left( \left( \phi_{I^c} \vee  \psi_{J^c} \right)  \vee  \al (\chi_{K^c})\right) \nonumber,
\end{align}
while on the other hand, we have
\begin{align}
  \Delta\left( \al(\phi) \vee (\psi \vee  \chi) \right)  = \sum_{\substack{I \subseteq \{1,\dots, n\}  \\ J \subseteq \{1,\dots, m\} \\K \subseteq \{1,\dots, k\}}}  \left( \al( \phi_I) \vee \left( \psi_J  \vee  \chi_K \right)\right)  \otimes \left( \al(\phi_{I^c}) \vee  \left( \psi_{J^c}   \vee   \chi_{K^c} \right)\right). \nonumber
\end{align}
Applying Lemma \ref{lem:coideal1} to $e= \left( \phi_I \vee \psi_J \right)\vee  \al (\chi_K) , f=\left( \phi_{I^c} \vee  \psi_{J^c} \right)  \vee  \al (\chi_{K^c}),
g=  \al( \phi_I) \vee \left( \psi_J  \vee  \chi_K \right), h= \al(\phi_{I^c}) \vee  \left( \psi_{J^c}   \vee   \chi_{K^c}\right)$,
we see that the difference between $ \Delta\left( (\phi \vee \psi) \vee \al (\chi) \right)$ and $ \Delta\left( \al(\phi) \vee (\psi \vee  \chi) \right) $
is an element in  ${\mathcal I} \otimes {\mathds{T}} + {\mathds{T}} \otimes {\mathcal I} $, which completes the proof.
\end{proof}

\begin{proof}[Proof of the proposition]
By definition of the ideal ${\mathcal I}$, the quotient space $\mathds{T}/{\mathcal I}$,  is a Hom-associative algebra
when equipped with $\vee$, and  ${\mathds{1}} $ is a unit. From Proposition \ref{prop:coalgebra}, it also follows that $({\mathds{T}}/{\mathcal I}, \Delta,\epsilon)$
is a coassociative algebra with counit $\epsilon$. According to Lemma \ref{Deltavee}, these induced structures are compatible.
\end{proof}

We now intend to define an antipode on the free Hom-associative algebra with $1$-generator. We first define $S: {\mathds{T}} \longrightarrow {\mathds{T}}$ by:
\begin{itemize}
  \item $S(\mathds{1})=\mathds{1} $;
  \item $S(|, a_1)= - (|, a_1)$, or $S\left( \begin{tikzpicture}[xscale=0.3, yscale=0.3,baseline={([yshift=-.8ex]current bounding box.center)}]
\draw[line width=1pt] (0,0) -- (0,1.5) ;
\draw (0,1.8) node {$a_1$};
\end{tikzpicture}\right) = - \begin{tikzpicture}[xscale=0.3, yscale=0.3,baseline={([yshift=-.8ex]current bounding box.center)}]
\draw[line width=1pt] (0,0) -- (0,1.5);
\draw (0,1.8) node {$a_1$};
\end{tikzpicture}$, where $a_1$ is a non-negative integer;
  \item $S$ is an antimorphism of  $({\mathds{T}},\vee)$, i.e., $S(\varphi \vee \psi)= S(\psi) \vee S(\varphi)$,  for any $\varphi, \psi \in B$.
\end{itemize}

Explicitly, $S$ maps a tree with $n$-leaves to $(-1)^n$ times the image of that tree through a vertical symmetry applied at each node, for example:
$$
S\left(
\begin{tikzpicture}[xscale=0.35, yscale=0.35,baseline={([yshift=-.8ex]current bounding box.center)}]
\draw[line width=1pt] (0,1) -- (0,1.5) -- (-1.5,3) -- (-2.5,4);
\draw[line width=1pt] (-2,3.5) -- (-1.5,4);
\draw[line width=1pt] (-1.5,3) -- (-0.5,4);
\draw[line width=1pt] (0,1.5) -- (1.75,4) ;
\draw[line width=1pt] (1.25,3.25) -- (0.5,4);


\draw (-2.5,4.5) node {$2$};
\draw (-1.5,4.5) node {$4$};
\draw (-0.5,4.5) node {$0$};
\draw (0.5,4.5) node {$3$};
\draw (1.75,4.5) node {$2$};
\end{tikzpicture}
 \right)= (-1)^5
 \begin{tikzpicture}[xscale=0.35, yscale=0.35,baseline={([yshift=-.8ex]current bounding box.center)}]
\draw[line width=1pt] (0,1) -- (0,1.5) -- (1.5,3) -- (2.5,4);
\draw[line width=1pt] (2,3.5) -- (1.5,4);
\draw[line width=1pt] (1.5,3) -- (0.5,4);
\draw[line width=1pt] (0,1.5) -- (-1.75,4) ;
\draw[line width=1pt] (-1.25,3.25) -- (-0.5,4);


\draw (2.5,4.5) node {$2$};
\draw (1.5,4.5) node {$4$};
\draw (0.5,4.5) node {$0$};
\draw (-0.5,4.5) node {$3$};
\draw (-1.75,4.5) node {$2$};
\end{tikzpicture}
$$

The map $S$ clearly goes to the quotient and induces an endomorphism of the free Hom-associative algebra with $1$-generator. The main result of this section is:

\begin{thm}\label{thm:HopfAlgOnTrees}
 The free Hom-associative algebra with $1$-generator is an $(\al,id )$-Hom-Hopf algebra
 (in the sense of Definition \ref{def:hom-Hopf-algebra_in_our_sense}) when equipped with the antipode $S$.
\end{thm}
\begin{proof}
It is obvious that $S$ commutes with $\al$ and  $S({\mathcal I}) \subset {\mathcal I}$, so $S$ goes to the quotient.
The only non-immediate result is that $S$ satisfies (\ref{eq:antipode_condition}).
The result is true for $\mathds{1}$ and for any element in $B_1$, i.e., of the form \begin{tikzpicture}[xscale=0.3, yscale=0.3,baseline={([yshift=-.8ex]current bounding box.center)}]
\draw[line width=1pt] (0,0) -- (0,1.5) ;
\draw (0,1.8) node {$a_1$};
\end{tikzpicture} where $a_1$ is a non-negative integer.
Let us prove that if the result holds true for leaf weighted trees $\varphi,\psi $ it also holds true for $\varphi \vee \psi $. Let $k$ be be the non-negative number such that
$\al^k \circ \vee \circ (S \otimes id) \circ \Delta (\varphi)=0$,  then
\begin{align}
 & \al^{k+1} \circ \vee \circ (S \otimes id) \circ \Delta (\varphi \vee \psi) = \nonumber \\
 & \qquad = \sum_{J,I} \al^{k+1} \circ \vee \circ (S \otimes id) (\varphi_J \vee \psi_I) \otimes (\varphi_{J^c} \vee \psi_{I c})   \nonumber \\
 & \qquad =  \sum_{J,I} \al^{k+1}\left( (S(\psi_I) \vee S(\varphi_J)) \vee (\varphi_{J^c} \vee \psi_{I c})\right)     \nonumber \\
& \qquad =  \sum_{J,I} \al^{k}\left( \al^2(S(\psi_I)) \vee \left(  \al (S(\varphi_J)) \vee (\varphi_{J^c} \vee \psi_{I c}) \right) \right), \mbox{ using Hom-associativity} \nonumber \\
& \qquad = \sum_{J,I} \al^{k} (\al^2(S(\psi_I)) \vee \left(  ( S(\varphi_J) \vee \varphi_{J^c}) \vee \al(\psi_{I c})  \right), \mbox{ using again Hom-associativity} \nonumber \\
& \qquad = 0. \nonumber
\end{align}
An induction step completes the proof.
\end{proof}

The antipode $S$ is an inverse in the usual sense for quite a few trees.
By inverse in the usual sense we mean that we can take $k=0$ in (\ref{eq:antipode_condition}).
Let us call \textbf{ferns} weighted trees whose underlying tree is such that each node is
related to a leaf. For instance, the following trees are ferns:
\begin{tikzpicture}[xscale=0.4, yscale=0.4,baseline={([yshift=-.8ex]current bounding box.center)}]
\draw[line width=1pt] (0,0) -- (0,1) -- (1,2);
\draw[line width=1pt] (0.3,1.3) -- (-0.2,2);
\draw[line width=1pt] (0.6,1.6) -- (0.35,2);
\draw[line width=1pt] (0,1) -- (-1,2);
\end{tikzpicture} and
\begin{tikzpicture}[xscale=0.4, yscale=0.4,baseline={([yshift=-.8ex]current bounding box.center)}]
\draw[line width=1pt] (0,0) -- (0,1) -- (1,2);
\draw[line width=1pt] (0.3,1.3) -- (-0.2,2);
\draw[line width=1pt] (0.07,1.6) -- (0.45,2);
\draw[line width=1pt] (0,1) -- (-1,2);
\end{tikzpicture} while
\begin{tikzpicture}[xscale=0.4, yscale=0.4,baseline={([yshift=-.8ex]current bounding box.center)}]
\draw[line width=1pt] (0,0) -- (0,1) -- (1,2);
\draw[line width=1pt] (0,1) -- (-1,2);
\draw[line width=1pt] (0.6,1.6) -- (0.2,2);
\draw[line width=1pt] (-0.6,1.6) -- (-0.2,2);
\end{tikzpicture}, is not a fern.


We still call \textbf{space of ferns} the subspace of the free vector  space $\mathds{T} $  generated by ferns and the subspace of the free Hom-associative algebra with $1$-generator which is the image of ferns in $\mathds{T} $ through the canonical projection from $ \mathds{T}  $
to the free Hom-associative algebra with $1$-generator.

Let us investigate the subspace of the free Hom-associative algebra with $1$-generator $\mathds{T}/{\mathcal I}$ made of all elements with invertibility index $0$, i.e. the subspace of all elements $g \in \mathds{T}/{\mathcal I}$ such that the following relation holds
\begin{equation} \label{eq:antipodeK0}
 \vee \circ (S \otimes \id) \circ \Delta (g)= \vee \circ( \id \otimes S) \circ \Delta (g)= \eta \circ \epsilon \, (g).
\end{equation}

\begin{prop}\label{prop:invert-index-ferns}
The subspace of the free Hom-associative algebra with $1$-generator $\mathds{T}/{\mathcal I}$ made of all elements with invertibility index $0$ is stable under left or right multiplication under elements in $B_1$.
It contains the space of ferns, as well as the spaces $B_i, i=1,2,3,4$.
\end{prop}
\begin{proof}
Let $\varphi$ be an element in $\mathds{T}$  where (\ref{eq:antipodeK0}) holds, then it also holds for any tree of the form \begin{tikzpicture}[xscale=0.4, yscale=0.4,baseline={([yshift=-.8ex]current bounding box.center)}]
\draw[line width=1pt] (0,0) -- (0,1.3) ;
\draw (0,1.5) node {$a_1$};
\end{tikzpicture} $\vee \varphi$ (in this case we need to use the Hom-associativity of $\vee$):

\begin{align}
  \vee \circ (S \otimes id ) \circ \Delta (\begin{tikzpicture}[xscale=0.4, yscale=0.4,baseline={([yshift=-.8ex]current bounding box.center)}]
\draw[line width=1pt] (0,0) -- (0,1.3) ;
\draw (0,1.5) node {$a_1$};
\end{tikzpicture} \vee \varphi) & = \vee \circ (S \otimes id ) \circ \left( \Delta (\begin{tikzpicture}[xscale=0.4, yscale=0.4,baseline={([yshift=-.8ex]current bounding box.center)}]
\draw[line width=1pt] (0,0) -- (0,1.3) ;
\draw (0,1.5) node {$a_1$};
\end{tikzpicture} ) \vee  \Delta(\varphi)\right) \nonumber \\
& = \vee \circ (S \otimes id ) \circ \left(  (\begin{tikzpicture}[xscale=0.4, yscale=0.4,baseline={([yshift=-.8ex]current bounding box.center)}]
\draw[line width=1pt] (0,0) -- (0,1.3) ;
\draw (0,1.5) node {$a_1$};
\end{tikzpicture} \otimes \mathds{1} + \mathds{1} \otimes \begin{tikzpicture}[xscale=0.4, yscale=0.4,baseline={([yshift=-.8ex]current bounding box.center)}]
\draw[line width=1pt] (0,0) -- (0,1.3) ;
\draw (0,1.5) node {$a_1$};
\end{tikzpicture} ) \vee (\varphi_J \otimes \varphi_{J^c})\right)  \nonumber \\
& = \vee \circ (S \otimes id ) \circ \left( \begin{tikzpicture}[xscale=0.4, yscale=0.4,baseline={([yshift=-.8ex]current bounding box.center)}]
\draw[line width=1pt] (0,0) -- (0,1) -- (1,2);
\draw[line width=1pt] (0,1) -- (-1,2);
\draw (1,2.2) node {$\varphi_J$};
\draw (-1,2.3) node {$\mathds{1}$};
\end{tikzpicture} \otimes \al(\varphi_{J^c})  + \al(\varphi_{J})   \otimes \begin{tikzpicture}[xscale=0.4, yscale=0.4,baseline={([yshift=-.8ex]current bounding box.center)}]
\draw[line width=1pt] (0,0) -- (0,1) -- (1,2);
\draw[line width=1pt] (0,1) -- (-1,2);
\draw (1,2.2) node {$\varphi_{J^c}$};
\draw (-1,2.3) node {$\mathds{1}$};
\end{tikzpicture} \right)  \nonumber
\end{align}
which proves the claim.
It turns out that  this  also proves that relations
 $ \vee \circ (S \otimes id) \circ \Delta (\varphi)= \mu \circ( id \otimes S) \circ \Delta (\varphi)=0$
hold for all ferns, since any fern in $B_{n+1}$ is obtained out of
a fern in $B_n$ by either left or right grafting with an element in $B_1$.
In particular, this relation also holds for all elements in $B_2,B_3$, since those spaces are included
in the space of ferns.
To check it for all elements in $B_4$, it suffices to check it in the unique element which is not a fern, i.e.
\begin{tikzpicture}[xscale=0.4, yscale=0.4,baseline={([yshift=-.8ex]current bounding box.center)}]
\draw[line width=1pt] (0,0) -- (0,1) -- (1,2);
\draw[line width=1pt] (0,1) -- (-1,2);
\draw[line width=1pt] (0.6,1.6) -- (0.2,2);
\draw[line width=1pt] (-0.6,1.6) -- (-0.2,2);
\end{tikzpicture},
 which is a direct computation.
\end{proof}

We give an example of an element in the free Hom-associative algebra with $1$-generator for which the antipode does not satisfy Relation (\ref{eq:antipodeK0}):
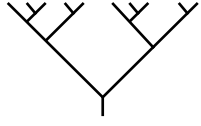
\begin{center}
 \begin{tikzpicture}[xscale=0.5, yscale=0.5,baseline={([yshift=-.8ex]current bounding box.center)}]
\draw[line width=1pt] (0,1) -- (0,1.5) -- (-1.5,3) -- (-2.5,4);
\draw[line width=1pt] ((0,1.5) -- (1.5,3) -- (2.5,4);
\draw[line width=1pt] (-2,3.5) -- (-1.5,4);
\draw[line width=1pt] (-1.75,3.7) -- (-2,4);
\draw[line width=1pt] (-1.5,3) -- (-0.5,4);
\draw[line width=1pt] (-0.75,3.7) -- (-1,4);
\draw[line width=1pt] (1.35,2.8) -- (0.25,4);
\draw[line width=1pt] (0.7,3.5) -- (1.2,4);
\draw[line width=1pt] (0.95,3.7) -- (0.7,4);
\draw[line width=1pt] (2.25,3.7) -- (2,4);
 \end{tikzpicture}
 \end{center}

Proof: Left to the reader.

\subsection{Universal enveloping Lie algebra of a Hom-Lie algebra}

Let $(\gg, \brr{\, , \, }, \al)$ be a multiplicative Hom-Lie algebra. Let us apply the so-called Schur functor \cite{Loday} to ${\mathds{T}} $ and $ \gg$. We define ${\mathds{T}}^\gg $ to be the vector space:
$$
 {\mathds{1}} \oplus \bigoplus_{n \geq 1}  B_n \otimes \gg^{\otimes n} .
 $$
Elements of $B_n \otimes \gg^{\otimes n}$ shall be pictured for every $ \phi \in B_n$, $x_1, \dots,x_n \in \gg$,
by inserting, for all $i=1, \dots,n$, the element $ x_i$ at the top of the leaf with label $x_i$:
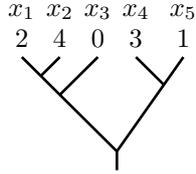
 $$ \begin{tikzpicture}[xscale=0.5, yscale=0.5,baseline={([yshift=-.8ex]current bounding box.center)}]
\draw[line width=1pt] (0,1) -- (0,1.5) -- (-1.5,3) -- (-2.5,4);
\draw[line width=1pt] (-2,3.5) -- (-1.5,4);
\draw[line width=1pt] (-1.5,3) -- (-0.5,4);
\draw[line width=1pt] (0,1.5) -- (1.75,4) ;
\draw[line width=1pt] (1.25,3.25) -- (0.5,4);


\draw (-2.5,4.5) node {$2$};
\draw (-1.5,4.5) node {$4$};
\draw (-0.5,4.5) node {$0$};
\draw (0.5,4.5) node {$3$};
\draw (1.75,4.5) node {$1$};
\draw (-2.5,5.2) node {$x_1$};
\draw (-1.5,5.2) node {$x_2$};
\draw (-0.5,5.2) node {$x_3$};
\draw (0.5,5.2) node {$x_4$};
\draw (1.75,5.2) node {$x_5$};
\end{tikzpicture}$$
From now, we only write down the weight of a given leaf of an element in   $B_n \otimes \gg^{\otimes n}$ when it is not equal to zero.
So \begin{tikzpicture}[xscale=0.5, yscale=0.5,baseline={([yshift=-.8ex]current bounding box.center)}]
\draw[line width=1pt] (0,0) -- (0,1) -- (1,2);
\draw[line width=1pt] (0.5,1.5) -- (0,2);
\draw[line width=1pt] (0,1) -- (-1,2);
\draw (0,2.5) node {$ 2$};
\end{tikzpicture}
is a short hand for\begin{tikzpicture}[xscale=0.5, yscale=0.5,baseline={([yshift=-.8ex]current bounding box.center)}]
\draw[line width=1pt] (0,0) -- (0,1) -- (1,2);
\draw[line width=1pt] (0.5,1.5) -- (0,2);
\draw[line width=1pt] (0,1) -- (-1,2);
\draw (1,2.5) node {$0$};
\draw (0,2.5) node {$ 2$};
\draw (-1,2.5) node {$0$};
\end{tikzpicture}.

The operations $ \vee,S,\Delta,\epsilon, \al$ defined on $ {\mathds{T}} $ have natural extensions
to ${\mathds{T}}^\gg $, that we denote by the same symbols. The extension $\Delta$ is still a coproduct with counit $\epsilon$, which is still compatible with $\vee $.
Moreover, the subspace $ {\mathcal I}^\gg = \oplus_{n \geq 3} {\mathcal I}_n \otimes \gg^{\otimes n} $,
(with $ {\mathcal I}_n $ being the subspace of ${\mathcal I}  \cap B_n $) is an ideal for $\vee$
 and a coideal for $\Delta$. It follows directly from Theorem \ref{thm:HopfAlgOnTrees} that the sextuple
$({\mathds{T}}^\gg/{\mathcal I}^\gg,\vee,\Delta \circ \al,S,\mathds{1},\epsilon) $ is an $(\al,id)$-Hom-Hopf algebra.
We call this Hom-Hopf algebra the \textbf{free Hom-associative algebra with generators in~$\gg$}.

We now consider the quotient of the latter Hom-Hopf algebra by the ideal $ {\mathcal J}^\gg$ of $({\mathds{T}}^\gg/{\mathcal I}^\gg,\vee,\mathds{1})$ generated by:
\begin{enumerate}
 \item[(i)] elements of the form \begin{tikzpicture}[xscale=0.4, yscale=0.4,baseline={([yshift=-.8ex]current bounding box.center)} ]
\draw[line width=1pt] (0,-0.5) -- (0,1.5) ;
\draw (0,1.8) node {$n$};
\draw (0,2.5) node {$x$};
\end{tikzpicture} - \begin{tikzpicture}[xscale=0.4, yscale=0.4,baseline={([yshift=-.8ex]current bounding box.center)}]
\draw[line width=1pt] (0,-0.5) -- (0,1.5) ;
\draw (0,2) node {$\al^n(x)$};
\end{tikzpicture},
(recall that for all $ y \in {\mathfrak g} $, \begin{tikzpicture}[xscale=0.4, yscale=0.4,baseline={([yshift=-.8ex]current bounding box.center)}]
\draw[line width=1pt] (0,-0.5) -- (0,1.5) ;
\draw (0,2) node {$y$};
\end{tikzpicture} is a short hand for \begin{tikzpicture}[xscale=0.4, yscale=0.4,baseline={([yshift=-.8ex]current bounding box.center)} ]
\draw[line width=1pt] (0,-0.5) -- (0,1.5) ;
\draw (0,1.95) node {$0$};
\draw (0,2.75) node {$y$};
\end{tikzpicture} )
 \item[(ii)] elements of the form \begin{tikzpicture}[xscale=0.4, yscale=0.4,baseline={([yshift=-.8ex]current bounding box.center)}]
\draw[line width=1pt] (0,0) -- (0,1) -- (1,2);
\draw[line width=1pt] (0,1) -- (-1,2);
\draw (-1,2.5) node {$x$};
\draw (1,2.5) node {$y$};
\end{tikzpicture} - \begin{tikzpicture}[xscale=0.4, yscale=0.4,baseline={([yshift=-.8ex]current bounding box.center)}]
\draw[line width=1pt] (0,0) -- (0,1) -- (1,2);
\draw[line width=1pt] (0,1) -- (-1,2);
\draw (-1,2.5) node {$y$};
\draw (1,2.5) node {$x$};
\end{tikzpicture} - \begin{tikzpicture}[xscale=0.4, yscale=0.4,baseline={([yshift=-.8ex]current bounding box.center)}]
\draw[line width=1pt] (0,-0.5) -- (0,1.5) ;
\draw (0,2) node {$[x,y]$};
\end{tikzpicture}.
\end{enumerate}

We first establish the following result.

\begin{prop}
The ideal $ {\mathcal J}^\gg$ is a Hom-Hopf ideal of
the free Hom-associative algebra with generators in $\gg$.
\end{prop}
\begin{proof}
Again, it suffices to show that the structural maps go to the quotient with respect to ${\mathcal J}^\gg $.
Since $\Delta $ and $\vee $ are compatible in the sense of equation (\ref{compatibility}), to show that $\Delta $ goes to the quotient suffices to show that
for all $ x,y \in \gg$, both
$ \Delta \left( \begin{tikzpicture}[xscale=0.4, yscale=0.4,baseline={([yshift=-.8ex]current bounding box.center)} ]
\draw[line width=1pt] (0,-0.5) -- (0,1.5) ;
\draw (0,1.8) node {$n$};
\draw (0,2.5) node {$x$};
\end{tikzpicture} - \begin{tikzpicture}[xscale=0.4, yscale=0.4,baseline={([yshift=-.8ex]current bounding box.center)}]
\draw[line width=1pt] (0,-0.5) -- (0,1.5) ;
\draw (0,2) node {$\al^n(x)$};
\end{tikzpicture} \right) $ and $\Delta \left( \begin{tikzpicture}[xscale=0.4, yscale=0.4,baseline={([yshift=-.8ex]current bounding box.center)}]
\draw[line width=1pt] (0,0) -- (0,1) -- (1,2);
\draw[line width=1pt] (0,1) -- (-1,2);
\draw (-1,2.5) node {$x$};
\draw (1,2.5) node {$y$};
\end{tikzpicture} - \begin{tikzpicture}[xscale=0.4, yscale=0.4,baseline={([yshift=-.8ex]current bounding box.center)}]
\draw[line width=1pt] (0,0) -- (0,1) -- (1,2);
\draw[line width=1pt] (0,1) -- (-1,2);
\draw (-1,2.5) node {$y$};
\draw (1,2.5) node {$x$};
\end{tikzpicture} - \begin{tikzpicture}[xscale=0.4, yscale=0.4,baseline={([yshift=-.8ex]current bounding box.center)}]
\draw[line width=1pt] (0,-0.5) -- (0,1.5) ;
\draw (0,2) node {$[x,y]$};
\end{tikzpicture} \right) $
are elements in
${\mathcal J}^\gg \otimes {\mathds{T}}^\gg/ {\mathcal{I}}^\gg + {\mathds{T}}^\gg/ {\mathcal{I}}^\gg\otimes {\mathcal J}^\gg$.
This follows from the following two computations, obtained by using directly the definition (\ref{def:coproduct}) of the coproduct :
$$\Delta\left( \begin{tikzpicture}[xscale=0.4, yscale=0.4,baseline={([yshift=-.8ex]current bounding box.center)} ]
\draw[line width=1pt] (0,-0.5) -- (0,1.5) ;
\draw (0,1.8) node {$n$};
\draw (0,2.5) node {$x$};
\end{tikzpicture} - \begin{tikzpicture}[xscale=0.4, yscale=0.4,baseline={([yshift=-.8ex]current bounding box.center)}]
\draw[line width=1pt] (0,-0.5) -- (0,1.5) ;
\draw (0,2) node {$\al^n(x)$};
\end{tikzpicture} \right)= \left( \begin{tikzpicture}[xscale=0.4, yscale=0.4,baseline={([yshift=-.8ex]current bounding box.center)} ]
\draw[line width=1pt] (0,-0.5) -- (0,1.5) ;
\draw (0,1.8) node {$n$};
\draw (0,2.5) node {$x$};
\end{tikzpicture} - \begin{tikzpicture}[xscale=0.4, yscale=0.4,baseline={([yshift=-.8ex]current bounding box.center)}]
\draw[line width=1pt] (0,-0.5) -- (0,1.5) ;
\draw (0,2) node {$\al^n(x)$};
\end{tikzpicture} \right)  \otimes  \mathds{1}  +  \mathds{1} \otimes \left( \begin{tikzpicture}[xscale=0.4, yscale=0.4,baseline={([yshift=-.8ex]current bounding box.center)} ]
\draw[line width=1pt] (0,-0.5) -- (0,1.5) ;
\draw (0,1.8) node {$n$};
\draw (0,2.5) node {$x$};
\end{tikzpicture} - \begin{tikzpicture}[xscale=0.4, yscale=0.4,baseline={([yshift=-.8ex]current bounding box.center)}]
\draw[line width=1pt] (0,-0.5) -- (0,1.5) ;
\draw (0,2) node {$\al^n(x)$};
\end{tikzpicture} \right) $$
and
$$\Delta \left( \begin{tikzpicture}[xscale=0.4, yscale=0.4,baseline={([yshift=-.8ex]current bounding box.center)}]
\draw[line width=1pt] (0,0) -- (0,1) -- (1,2);
\draw[line width=1pt] (0,1) -- (-1,2);
\draw (-1,2.5) node {$x$};
\draw (1,2.5) node {$y$};
\end{tikzpicture} - \begin{tikzpicture}[xscale=0.4, yscale=0.4,baseline={([yshift=-.8ex]current bounding box.center)}]
\draw[line width=1pt] (0,0) -- (0,1) -- (1,2);
\draw[line width=1pt] (0,1) -- (-1,2);
\draw (-1,2.5) node {$y$};
\draw (1,2.5) node {$x$};
\end{tikzpicture} - \begin{tikzpicture}[xscale=0.4, yscale=0.4,baseline={([yshift=-.8ex]current bounding box.center)}]
\draw[line width=1pt] (0,-0.5) -- (0,1.5) ;
\draw (0,2) node {$[x,y]$};
\end{tikzpicture} \right)= \left( \begin{tikzpicture}[xscale=0.4, yscale=0.4,baseline={([yshift=-.8ex]current bounding box.center)}]
\draw[line width=1pt] (0,0) -- (0,1) -- (1,2);
\draw[line width=1pt] (0,1) -- (-1,2);
\draw (-1,2.5) node {$x$};
\draw (1,2.5) node {$y$};
\end{tikzpicture} - \begin{tikzpicture}[xscale=0.4, yscale=0.4,baseline={([yshift=-.8ex]current bounding box.center)}]
\draw[line width=1pt] (0,0) -- (0,1) -- (1,2);
\draw[line width=1pt] (0,1) -- (-1,2);
\draw (-1,2.5) node {$y$};
\draw (1,2.5) node {$x$};
\end{tikzpicture} - \begin{tikzpicture}[xscale=0.4, yscale=0.4,baseline={([yshift=-.8ex]current bounding box.center)}]
\draw[line width=1pt] (0,-0.5) -- (0,1.5) ;
\draw (0,2) node {$[x,y]$};
\end{tikzpicture} \right) \otimes  \mathds{1}  +  \mathds{1} \otimes \left( \begin{tikzpicture}[xscale=0.4, yscale=0.4,baseline={([yshift=-.8ex]current bounding box.center)}]
\draw[line width=1pt] (0,0) -- (0,1) -- (1,2);
\draw[line width=1pt] (0,1) -- (-1,2);
\draw (-1,2.5) node {$x$};
\draw (1,2.5) node {$y$};
\end{tikzpicture} - \begin{tikzpicture}[xscale=0.4, yscale=0.4,baseline={([yshift=-.8ex]current bounding box.center)}]
\draw[line width=1pt] (0,0) -- (0,1) -- (1,2);
\draw[line width=1pt] (0,1) -- (-1,2);
\draw (-1,2.5) node {$y$};
\draw (1,2.5) node {$x$};
\end{tikzpicture} - \begin{tikzpicture}[xscale=0.4, yscale=0.4,baseline={([yshift=-.8ex]current bounding box.center)}]
\draw[line width=1pt] (0,-0.5) -- (0,1.5) ;
\draw (0,2) node {$[x,y]$};
\end{tikzpicture} \right).$$
The antipode $S$ being an antimorphism of $\vee $, it suffices to check that it preserves the set of generators of ${\mathcal J}^\gg $ to state that preserves ${\mathcal J}^\gg$.
This simply follows from the definition of $S$, since:
 $$ S\left( \begin{tikzpicture}[xscale=0.4, yscale=0.4,baseline={([yshift=-.8ex]current bounding box.center)} ]
\draw[line width=1pt] (0,-0.5) -- (0,1.5) ;
\draw (0,1.8) node {$n$};
\draw (0,2.5) node {$x$};
\end{tikzpicture} - \begin{tikzpicture}[xscale=0.4, yscale=0.4,baseline={([yshift=-.8ex]current bounding box.center)}]
\draw[line width=1pt] (0,-0.5) -- (0,1.5) ;
\draw (0,2) node {$\al^n(x)$};
\end{tikzpicture} \right) =  -\left( \begin{tikzpicture}[xscale=0.4, yscale=0.4,baseline={([yshift=-.8ex]current bounding box.center)} ]
\draw[line width=1pt] (0,-0.5) -- (0,1.5) ;
\draw (0,1.8) node {$n$};
\draw (0,2.5) node {$x$};
\end{tikzpicture} - \begin{tikzpicture}[xscale=0.4, yscale=0.4,baseline={([yshift=-.8ex]current bounding box.center)}]
\draw[line width=1pt] (0,-0.5) -- (0,1.5) ;
\draw (0,2) node {$\al^n(x)$};
\end{tikzpicture} \right) $$
and\footnote{Notice that this is the first time that we are using the skew-symmetry of the bracket $[.,.]$ on ${\mathfrak g}$.} $S\left( \begin{tikzpicture}[xscale=0.4, yscale=0.4,baseline={([yshift=-.8ex]current bounding box.center)}]
\draw[line width=1pt] (0,0) -- (0,1) -- (1,2);
\draw[line width=1pt] (0,1) -- (-1,2);
\draw (-1,2.5) node {$x$};
\draw (1,2.5) node {$y$};
\end{tikzpicture} - \begin{tikzpicture}[xscale=0.4, yscale=0.4,baseline={([yshift=-.8ex]current bounding box.center)}]
\draw[line width=1pt] (0,0) -- (0,1) -- (1,2);
\draw[line width=1pt] (0,1) -- (-1,2);
\draw (-1,2.5) node {$y$};
\draw (1,2.5) node {$x$};
\end{tikzpicture} - \begin{tikzpicture}[xscale=0.4, yscale=0.4,baseline={([yshift=-.8ex]current bounding box.center)}]
\draw[line width=1pt] (0,-0.5) -- (0,1.5) ;
\draw (0,2) node {$[x,y]$};
\end{tikzpicture} \right) = - \left( \begin{tikzpicture}[xscale=0.4, yscale=0.4,baseline={([yshift=-.8ex]current bounding box.center)}]
\draw[line width=1pt] (0,0) -- (0,1) -- (1,2);
\draw[line width=1pt] (0,1) -- (-1,2);
\draw (-1,2.5) node {$x$};
\draw (1,2.5) node {$y$};
\end{tikzpicture} - \begin{tikzpicture}[xscale=0.4, yscale=0.4,baseline={([yshift=-.8ex]current bounding box.center)}]
\draw[line width=1pt] (0,0) -- (0,1) -- (1,2);
\draw[line width=1pt] (0,1) -- (-1,2);
\draw (-1,2.5) node {$y$};
\draw (1,2.5) node {$x$};
\end{tikzpicture} - \begin{tikzpicture}[xscale=0.4, yscale=0.4,baseline={([yshift=-.8ex]current bounding box.center)}]
\draw[line width=1pt] (0,-0.5) -- (0,1.5) ;
\draw (0,2) node {$[x,y]$};
\end{tikzpicture} \right) .$

This completes the proof.
\end{proof}

\begin{defn}\label{def:Hom-Universal}
The {\bf universal enveloping algebra of a  multiplicative Hom-Lie algebra $(\gg, \brr{\, , \, }, \al)$}
is by definition the quotient of the free Hom-associative algebra with generators in $\gg$ (which is an $(\al,\id)$-Hom-Hopf algebra) by the Hom-Hopf ideal ${\mathcal J}^\gg $. This quotient is itself an $(\al,\id)$-Hom-Hopf algebra that we denote by ${\mathcal U}\gg $, while we keep the usual symbols for its structural maps.
\end{defn}

\begin{ex}
 For $\al=id$, Hom-Lie algebras are just Lie algebras. It is routine to check that the universal enveloping algebra of $(\gg, \brr{\, , \, }, id) $ coincides with the usual universal enveloping algebra.

For $\al=0$ and $[.,.]$ an arbitrary skew-symmetric map, all weighted trees are in the ideal ${\mathcal I}^{\mathfrak g}$ unless the weight of each leaf is $0$ and the universal enveloping algebra is obtained by applying the Schur functor to the algebra of all binary trees, then dividing the outcome by the ideal (for grafting) generated by
$$
 \begin{tikzpicture}[xscale=0.4, yscale=0.4,baseline={([yshift=-.8ex]current bounding box.center)}]
\draw[line width=1pt] (0,0) -- (0,1) -- (1,2);
\draw[line width=1pt] (0,1) -- (-1,2);
\draw (-1,2.5) node {$x$};
\draw (1,2.5) node {$y$};
\end{tikzpicture} - \begin{tikzpicture}[xscale=0.4, yscale=0.4,baseline={([yshift=-.8ex]current bounding box.center)}]
\draw[line width=1pt] (0,0) -- (0,1) -- (1,2);
\draw[line width=1pt] (0,1) -- (-1,2);
\draw (-1,2.5) node {$y$};
\draw (1,2.5) node {$x$};
\end{tikzpicture} - \begin{tikzpicture}[xscale=0.4, yscale=0.4,baseline={([yshift=-.8ex]current bounding box.center)}]
\draw[line width=1pt] (0,-0.5) -- (0,1.5) ;
\draw (0,2) node {$[x,y]$};
\end{tikzpicture}
$$
for all $x,y \in {\mathfrak g}$. This is of course not associative (it is by construction Hom-associative with respect to a map that satisfies $\al=0$ on (the image of) $\bigoplus_{n \geq 1}  B_n \otimes \gg^{\otimes n}$, which is not a strong constraint). Also, the coproduct is simply:
$ \Delta (\phi) = \phi \otimes {\mathds{1}}+ {\mathds{1}} \otimes \phi$ for all $ \phi \in \bigoplus_{n \geq 1}  B_n \otimes \gg^{\otimes n}$.
\end{ex}

\begin{rem}\label{rem:functor}
Any Hom-Lie algebra morphism $\psi:\gg \to \gg'$ induces a natural $(\al,\id)$-Hom-Hopf algebra morphism
${\mathcal U}\phi : {\mathcal U}\gg  \to {\mathcal U}\gg'$ by:
 $$ {\mathcal U}\phi : \phi \otimes (x_1 \otimes \dots \otimes x_n ) \mapsto \phi \otimes (\psi(x_1) \otimes \dots \otimes \psi(x_n) ) $$
for all weighted $n$-tree $\phi$ and $x_1, \dots,x_n \in \gg$.
Associating to a Hom-Lie algebra $\gg$ a universal  envelopping algebra ${\mathcal U}\gg$,
one therefore obtains a functor ${\mathcal U} $ from the category of Hom-Lie algebras to the category of
$(\al,\id)$-Hom-Hopf algebras.
\end{rem}

 Notice that for Hom-Lie algebras constructed by composition out of a Lie algebra $\gg $, equipped with a bracket $[.,.]_{Lie}$, through an endomorphism $\al$, there are two natural Hom-Hopf algebra structures:
\emph{(i)} the universal enveloping algebra of the Hom-Lie algebra $ (\gg,\al \circ [.,.],\al)$
as in Definiton \ref{def:Hom-Universal} and \emph{(ii)} the $(\tilde{\al},\id)$-twist of the universal enveloping algebra (in the usual sense) $U^{Lie}(\gg) $
 of the Lie algebra $ (\gg, [.,.]_{Lie})$
through the Hopf algebra morphism $\tilde{\al} $ associated to $\al $ (this is an $(\al,\id) $-Hopf-algebra by Example \ref{Hom-Hopf-Twist}).

These two Hom-Hopf algebra do not agree. This is easily seen when $\al=0$, for
the product in simply $0$ for the Hom-Hopf algebra of item \emph{(ii)} (since its product is $\mu_{\al}=\al \circ \mu$
with $\mu$ the product of $U^{Lie}(\gg)$), while it is not zero for
the universal enveloping algebra of the Hom-Lie algebra $ (\gg,\al \circ [.,.],\al)$.
For instance, the product
\begin{center}
\begin{tikzpicture}[xscale=0.4, yscale=0.4,baseline={([yshift=-.8ex]current bounding box.center)}]
\draw[line width=1pt] (0,-0.5) -- (0,1.5) ;
\draw (0,2) node {$x$};
\end{tikzpicture}
$\vee$
\begin{tikzpicture}[xscale=0.4, yscale=0.4,baseline={([yshift=-.8ex]current bounding box.center)}]
\draw[line width=1pt] (0,-0.5) -- (0,1.5) ;
\draw (0,2) node {$x$};
\end{tikzpicture}
\end{center}
does not vanish if $ x \in {\mathfrak g}$ a non-zero element. Notice that the spaces of which they are defined also do not coincide.

\subsection{Primitive elements on the free Hom-associative algebra with $1$-generator}

Besides the properties of primitive elements of an $(\al,\beta)$-Hom-Hopf algebra described in Remark \ref{rmk:SeveralPoints}, we aim to discuss the  case of $\mathds{T}/{\mathcal I}$.
By construction, $\al :\mathds{T} \to \mathds{T}$ is injective, and it is natural to ask
if its induced map $\al$ on the free Hom-associative algebra with $1$-generator $\mathds{T}/{\mathcal I}$ is injective as well. The answer is negative, in view of the next proposition, which introduces  a useful  element in building  counterexamples. It shows in particular that although all elements in $B_1$ (i.e. leaf weighted $1$-trees) are primitive for the coassociative coproduct (defined as in second item of Remark \ref{rmk:SeveralPoints}), the transpose is not true: there are primitive elements which are not in $B_1$.

\begin{prop}\label{prop:remarquable}
Consider the following element in  $\mathds{T}$:
$$
\begin{tikzpicture}[xscale=0.6, yscale=0.6, baseline={([yshift=-.8ex]current bounding box.center)}]
\draw[line width=1pt] (0,0.5) -- (0,1) -- (1,2);
\draw[line width=1pt] (0,1) -- (-1,2);
\draw[line width=1pt] (0.6,1.6) -- (0.2,2);
\draw[line width=1pt] (-0.6,1.6) -- (-0.2,2);
\draw (-1,2.3) node {$0$};
\draw (-0.2,2.3) node {$1$};
\draw (0.2,2.3) node {$1$};
\draw (1,2.3) node {$0$};
\end{tikzpicture} - \begin{tikzpicture}[xscale=0.6, yscale=0.6,baseline={([yshift=-.8ex]current bounding box.center)}]
\draw[line width=1pt] (0,0.5) -- (0,1) -- (1,2);
\draw[line width=1pt] (0.3,1.3) -- (-0.2,2);
\draw[line width=1pt] (0.07,1.6) -- (0.45,2);
\draw[line width=1pt] (0,1) -- (-1,2);
\draw (-1,2.3) node {$1$};
\draw (-0.2,2.3) node {$0$};
\draw (0.45,2.3) node {$0$};
\draw (1,2.3) node {$0$};
\end{tikzpicture}.
$$
Denote by $u$ its class in the free Hom-associative algebra with $1$-generator  $\mathds{T}/{\mathcal I}$.
Then $u \neq 0$ but $\al (u) =0$. Moreover, $u$ is a primitive element, and any element in the algebra (for $\vee$) generated by $u $ is primitive.
\end{prop}
\begin{proof}
It is clear that the tree  that defines $u$ is not contained in $ {\mathcal I}$, which is therefore not equal to zero.
Also, $\al(u)= \begin{tikzpicture}[xscale=0.6, yscale=0.6, baseline={([yshift=-.8ex]current bounding box.center)}]
\draw[line width=1pt] (0,0.5) -- (0,1) -- (1,2);
\draw[line width=1pt] (0,1) -- (-1,2);
\draw[line width=1pt] (0.6,1.6) -- (0.2,2);
\draw[line width=1pt] (-0.6,1.6) -- (-0.2,2);
\draw (-1,2.3) node {$1$};
\draw (-0.2,2.3) node {$2$};
\draw (0.2,2.3) node {$2$};
\draw (1,2.3) node {$1$};
\end{tikzpicture} - \begin{tikzpicture}[xscale=0.6, yscale=0.6,baseline={([yshift=-.8ex]current bounding box.center)}]
\draw[line width=1pt] (0,0.5) -- (0,1) -- (1,2);
\draw[line width=1pt] (0.3,1.3) -- (-0.2,2);
\draw[line width=1pt] (0.07,1.6) -- (0.45,2);
\draw[line width=1pt] (0,1) -- (-1,2);
\draw (-1,2.3) node {$2$};
\draw (-0.2,2.3) node {$1$};
\draw (0.45,2.3) node {$1$};
\draw (1,2.3) node {$1$};
\end{tikzpicture}=0$, using Hom-associativity. By a direct computation, $\Delta(u)= u \otimes  \mathds{1}  + \mathds{1} \otimes  u$, i.e. $u$ is a primitive element.

As a consequence, $ u$ is a primitive element contained in the kernel of $ \al$.
Now, for every pair $ u_1,u_2$ of primitive elements contained in the kernel of $\al$,
$u_1 \vee u_2$ is a primitive element, as follows from  equation (\ref{compatibility}):
 \begin{eqnarray*}  \Delta (u_1 \vee  u_2 )  &=&  \Delta (u_1 )  \vee \Delta (u_2)  \\
 &=& (u_1 \otimes  \mathds{1}  +  \mathds{1}  \otimes  u_1) \otimes ( u_2 \otimes  \mathds{1}  +  \mathds{1}  \otimes  u_2)  \\
& = & (u_1 \vee u_2) \otimes  \mathds{1}  +  \mathds{1}  \otimes (u_1 \vee u_2)  \\
&   & + (  u_1 \vee  \mathds{1}  ) \otimes  ( \mathds{1}  \vee u_2 ) +( \mathds{1}  \vee u_2) \otimes (u_1 \vee  \mathds{1} ) \\
& = &  (u_1 \vee u_2) \otimes  \mathds{1}  +  \mathds{1}  \otimes (u_1 \vee u_2) \\
& &  +  \al(  u_1) \otimes  \al( u_2 ) + \al (u_2) \otimes \al (u_1) \\
 & =&   (u_1 \vee u_2) \otimes  \mathds{1}  +  \mathds{1}  \otimes (u_1 \vee u_2) \end{eqnarray*}
Moreover, Equation (\ref{alphavee}) implies that any element in the algebra generated by $u $ is in the kernel of $\al$. Altogether, these properties imply that the space of elements with are primitive and in the kernel of $\al$
is stable under $\vee$. In particular, every element in the algebra generated by $u$ is primitive.
\end{proof}

\begin{rem}
 In view of Remark \ref{rem:universal}, Proposition \ref{prop:remarquable} implies that for any Hom-associative algebra $(\A,\vee, \al)$, and any
 $x,y,z,t \in \A$, the element
  $$ \al(t) \vee ((x \vee y) \vee t) - (t\vee \al(x)) \vee (\al(y) \vee z)$$
is  in the kernel of $\al$.
\end{rem}

\subsection{Canonical $n$-ary operations on Hom-associative algebras}
Let $A$ be a Hom-associative algebra.
On  $ \A$, making the product of $n$ elements, $n \geq 3$  depends on the order in which the products are taken.  There is however a natural manner manner  to define $n$-ary operations $\A^{\otimes n} \to \A $, as we will see in the sequel. Recall that by Remark \ref{rem:universal}, operations of $n$ elements of $A$ are encoded by leaf weighted $n$-trees.

Given a  $n$-tree $ \varphi \in T_n $ (without weights) and an integer $n \leq k$,
we define a leaf weighted $n$-tree $\varphi[k] $ by assigning  to the leaf with label $i$ the integer
$ k - \ell (i)$ with $ \ell$ being the length of the branch from the root to the leaf. The length of the leaf $i$ is then the length of the path from the root to the leaf,
when the tree is seen as a graph.
For instance, let
$\varphi=\begin{tikzpicture}[xscale=0.4, yscale=0.4,baseline={([yshift=-.8ex]current bounding box.center)}]
\draw[line width=1pt] (0,1) -- (0,1.5) -- (-1.5,3) -- (-2.5,4);
\draw[line width=1pt] (-2,3.5) -- (-1.5,4);
\draw[line width=1pt] (-1.5,3) -- (-0.5,4);
\draw[line width=1pt] (0,1.5) -- (1.75,4) ;
\draw[line width=1pt] (1.25,3.25) -- (0.5,4);
\end{tikzpicture}
$ then
$\varphi[7]=\begin{tikzpicture}[xscale=0.4, yscale=0.4,baseline={([yshift=-.8ex]current bounding box.center)}]
\draw[line width=1pt] (0,1) -- (0,1.5) -- (-1.5,3) -- (-2.5,4);
\draw[line width=1pt] (-2,3.5) -- (-1.5,4);
\draw[line width=1pt] (-1.5,3) -- (-0.5,4);
\draw[line width=1pt] (0,1.5) -- (1.75,4) ;
\draw[line width=1pt] (1.25,3.25) -- (0.5,4);
\draw (-2.5,4.5) node {$3$};
\draw (-1.5,4.5) node {$3$};
\draw (-0.5,4.5) node {$4$};
\draw (0.5,4.5) node {$4$};
\draw (1.75,4.5) node {$4$};
\end{tikzpicture}
$ and $\varphi[5]=\begin{tikzpicture}[xscale=0.4, yscale=0.4,baseline={([yshift=-.8ex]current bounding box.center)}]
\draw[line width=1pt] (0,1) -- (0,1.5) -- (-1.5,3) -- (-2.5,4);
\draw[line width=1pt] (-2,3.5) -- (-1.5,4);
\draw[line width=1pt] (-1.5,3) -- (-0.5,4);
\draw[line width=1pt] (0,1.5) -- (1.75,4) ;
\draw[line width=1pt] (1.25,3.25) -- (0.5,4);
\draw (-2.5,4.5) node {$1$};
\draw (-1.5,4.5) node {$1$};
\draw (-0.5,4.5) node {$2$};
\draw (0.5,4.5) node {$2$};
\draw (1.75,4.5) node {$2$};
\end{tikzpicture}
$.

Let us call \textbf{right $n$-fern} the $n$-tree (without weights) obtained by successive graftings on the right of the $1$-tree and denote it by $F_n^r$, i.e. $$F_n^r= \left(\left(\dots \left( \left(\begin{tikzpicture}[xscale=0.3, yscale=0.3,baseline={([yshift=-.8ex]current bounding box.center)}]
\draw[line width=1pt] (0,0) -- (0,1.5) ;
\end{tikzpicture} \vee \begin{tikzpicture}[xscale=0.3, yscale=0.3,baseline={([yshift=-.8ex]current bounding box.center)}]
\draw[line width=1pt] (0,0) -- (0,1.5) ;
\end{tikzpicture}\right) \vee  \begin{tikzpicture}[xscale=0.3, yscale=0.3,baseline={([yshift=-.8ex]current bounding box.center)}]
\draw[line width=1pt] (0,0) -- (0,1.5) ;
\end{tikzpicture} \right) \vee \dots \vee  \begin{tikzpicture}[xscale=0.3, yscale=0.3,baseline={([yshift=-.8ex]current bounding box.center)}]
\draw[line width=1pt] (0,0) -- (0,1.5) ;
\end{tikzpicture} \right) \vee \begin{tikzpicture}[xscale=0.3, yscale=0.3,baseline={([yshift=-.8ex]current bounding box.center)}]
\draw[line width=1pt] (0,0) -- (0,1.5) ;
\end{tikzpicture}\right)= \begin{tikzpicture}[xscale=0.4, yscale=0.4,baseline={([yshift=-.8ex]current bounding box.center)}]
\draw[line width=1pt] (0,1) -- (0,1.5) -- (-1.5,3) -- (-2.5,4);
\draw[line width=1pt] (-2,3.5) -- (-1.5,4);
\draw[line width=1pt] (-1.5,3) -- (-0.5,4);
\draw[line width=1pt] (0,1.5) -- (2.2,4) ;
\draw[line width=1pt] (-0.3,1.75) -- (1.65,4) ;
\draw (0.4,4) node {$\dots$};
\end{tikzpicture}$$
Of course, the right $n$-fern is a fern. Similarly we  call \textbf{left $n$-fern} the $n$-tree (without weights) obtained by successive graftings on the left of the $1$-tree and denote it by $F_n^l$, i.e. $$F_n^l= \left( \begin{tikzpicture}[xscale=0.3, yscale=0.3,baseline={([yshift=-.8ex]current bounding box.center)}]
\draw[line width=1pt] (0,0) -- (0,1.5) ;
\end{tikzpicture} \vee \left( \begin{tikzpicture}[xscale=0.3, yscale=0.3,baseline={([yshift=-.8ex]current bounding box.center)}]
\draw[line width=1pt] (0,0) -- (0,1.5) ;
\end{tikzpicture} \vee \dots \vee \left(  \begin{tikzpicture}[xscale=0.3, yscale=0.3,baseline={([yshift=-.8ex]current bounding box.center)}]
\draw[line width=1pt] (0,0) -- (0,1.5) ;
\end{tikzpicture} \vee \left( \begin{tikzpicture}[xscale=0.3, yscale=0.3,baseline={([yshift=-.8ex]current bounding box.center)}]
\draw[line width=1pt] (0,0) -- (0,1.5) ;
\end{tikzpicture}  \vee   \begin{tikzpicture}[xscale=0.3, yscale=0.3,baseline={([yshift=-.8ex]current bounding box.center)}]
\draw[line width=1pt] (0,0) -- (0,1.5) ;
\end{tikzpicture} \right) \right) \dots \right) \right)
= \begin{tikzpicture}[xscale=0.4, yscale=0.4,baseline={([yshift=-.8ex]current bounding box.center)}]
\draw[line width=1pt] (0,1) -- (0,1.5) -- (1.5,3) -- (2.5,4);
\draw[line width=1pt] (2,3.5) -- (1.5,4);
\draw[line width=1pt] (1.5,3) -- (0.5,4);
\draw[line width=1pt] (0,1.5) -- (-2.2,4) ;
\draw[line width=1pt] (0.3,1.75) -- (-1.65,4) ;
\draw (-0.4,4) node {$\dots$};
\end{tikzpicture}$$

We first prove a lemma:

\begin{lem}\label{KFougeres}
 Let $k,n $ be non-negative integers with  $ n \leq k$. The identity
 $$
 F_n^l[k]= F_n^r[k]
 $$
holds in the free Hom-associative algebra with 1-generator  ${\mathds{T}}/{\mathcal I}  $.
\end{lem}

\begin{proof}
This follows by a finite induction using   definitions of these trees and Hom-associativity.
\end{proof}
The following Lemma is straightforward.
\begin{lem}\label{lemmaKtoK-1}
Let  $\varphi_1 \in T_p, \varphi_2 \in T_q $ be two trees and $k$ an integer such that $k\geq p+q$.  Then
$$
(\varphi_1\vee \varphi_2)[k]= \varphi_1[k-1] \vee \varphi_2[k-1].
$$
\end{lem}
Now, we state the main result of this section showing that the operations on Hom-associative algebras encoded by the weighted $n$-trees $\varphi[k]$ depends only on $n$ and $k$.
\begin{prop}
\label{prop:indifference}
 Let $k,n $ be two non-negative integers with  $ n \leq k$. For any $n$-trees $\varphi,\psi  \in T_n$, the identity  $\varphi[k] = \psi[k]$ holds in the free Hom-associative algebra with 1-generator  ${\mathds{T}}/{\mathcal I}  $.
\end{prop}
\begin{proof}
It suffices to show that, for any tree $\varphi \in T_n$, we have $\varphi[k]= F_n^r[k]$. For $n=1,2,3$, it is routine to check that this identity is true (using Hom-associativity in case $n=3$). Let us suppose now that this equality holds for any tree in $T_p$ with $p< n$. Let  $\varphi$ be a tree in $T_n$ with $n>3$, then there exist $\varphi_1 \in T_p, \varphi_2 \in T_q $ such that $\varphi= \varphi_1 \vee \varphi_2$ and $p+q=n$. By Lemma \ref{lemmaKtoK-1} we have
$$
\varphi[k]= \varphi_1[k-1] \vee \varphi_2[k-1]
$$
and applying the hypothesis $\varphi_1[k-1]= F^r_p[k-1]$ and $\varphi_2[k-1]= F^r_q[k-1]=F^l_q[k-1]$, using also the Lemma \ref{KFougeres}. If $q=1$, then $\varphi[k]=F^r_p[k-1] \vee \begin{tikzpicture}[xscale=0.3, yscale=0.3,baseline={([yshift=-.8ex]current bounding box.center)}]
\draw[line width=1pt] (0,0) -- (0,1.5) ;
\draw (0,1.7) node {\tiny $k-2$};
\end{tikzpicture} = F^r_n[k] $, if not, using in each step Hom-associativity, we have  $\varphi[k]=F^r_p[k-1] \vee F^l_q[k-1] = F^r_{p+1}[k-1] \vee F^l_{q-1}[k-1]= F^r_{p+2}[k-1] \vee F^l_{q-2}[k-1]= \dots=F^r_{p+q-1}[k-1] \vee  \begin{tikzpicture}[xscale=0.3, yscale=0.3,baseline={([yshift=-.8ex]current bounding box.center)}]
\draw[line width=1pt] (0,0) -- (0,1.5) ;
\draw (0,1.7) node {\tiny $k-2$};
\end{tikzpicture}= F^r_{p+q}[k] = F^r_{n}[k]$.
\end{proof}

As a consequence, for every non-negative integers $k,n $  with  $ n \leq k$, we denote by  $\lfloor e^n \rfloor_k  $ the element  in  ${\mathds{T}}/{\mathcal I}  $, defined by  $\varphi[k] \in {\mathds{T}}/{\mathcal I} $ for an arbitrary $n$-tree $ \varphi \in T_n$. It is called the \textbf{$k$-weighted $n$-ary product}.

\begin{ex} $ \lfloor e^1 \rfloor_k= $\begin{tikzpicture}[xscale=0.4, yscale=0.4,baseline={([yshift=-.8ex]current bounding box.center)}]
\draw[line width=1pt] (0,0) -- (0,2);
\draw (0,2.5) node {$k-1$};
\end{tikzpicture} , $\lfloor e^2 \rfloor_k= $\begin{tikzpicture}[xscale=0.5, yscale=0.5,baseline={([yshift=-.8ex]current bounding box.center)}]
\draw[line width=1pt] (0,0) -- (0,1) -- (1,2);
\draw[line width=1pt] (0,1) -- (-1,2);
\draw (1,2.2) node {\scriptsize$k-2$};
\draw (-1,2.2) node {\scriptsize$k-2$};
\end{tikzpicture} , $\lfloor e^3 \rfloor_k=$ \begin{tikzpicture}[xscale=1, yscale=1,baseline={([yshift=-.8ex]current bounding box.center)}]
\draw[line width=1pt] (0,0) -- (0,1) -- (1,2);
\draw[line width=1pt] (0.5,1.5) -- (0,2);
\draw[line width=1pt] (0,1) -- (-1,2);
\draw (1,2.1) node {\scriptsize$k-3$};
\draw (0,2.1) node {\scriptsize$ k-3$};
\draw (-1,2.1) node {\scriptsize$ k-2$};
\end{tikzpicture}.
\end{ex}
\begin{lem}\label{lem:coproduct-explicit}
 For all non-negative integers $k,n,m $ with  $ n < k$ and $ m < k$, we have
\begin{enumerate}
\item
 $ \Delta ( \lfloor e^n \rfloor_k) = \sum_{i=0}^n
\left(
\begin{matrix}
  n   \\
i
\end{matrix}
\right)
 \lfloor e^i\rfloor_k \otimes \lfloor e^{n-i}\rfloor_{k} ,$
\item
   $ \lfloor e^n\rfloor_k \vee \lfloor e^m\rfloor _k = \lfloor e^{n+m}\rfloor_{k+1} =\al (\lfloor e^{n+m}\rfloor_k), $
   \item $S(\lfloor e^{n}\rfloor_{k} )=(-1)^n\lfloor e^{n}\rfloor_{k}.$
   \end{enumerate}
\end{lem}

\

Given a formal power series in one variable with real coefficients $ f(\nu)= \sum_{i \geq 0}^\infty a_i \nu^i$, we denote by $  \widehat{f}_p (\nu)$ and call it  \textbf{ $k$-weighted  realization of $f$ } the element in  ${\mathds{T}}/{\mathcal I}[[\nu]]$ modulo $\nu^{p+1}$ given by:
\begin{equation}\label{ChapFnu}
\widehat{f}_p(\nu) =a_0 \mathds{1}+ \sum_{i \geq 1}^p a_i \nu^i  \lfloor e^i \rfloor_p.
\end{equation}
We call   the sequence $(\widehat{f}_p(\nu))_{p\in \mathbb{N}}$  the \textbf{   realization of $f$ } and denote it by $\widehat{f}(\nu)$.\\

We provide the following properties:

\begin{prop}\label{Prop:proprietiesWidehat}
Given two  formal power series in one variable with real coefficients $ f(\nu)= \sum_{i \geq 1}^\infty a_i \nu^i$,
$ g(\nu)= \sum_{i \geq 1}^\infty b_i \nu^i$. We have:
\begin{enumerate}
  \item $\widehat{f}(\nu)\vee \widehat{g}(\nu)=\al( \widehat{fg}(\nu)),$
  \item for all $k\in \mathbb{N}$, $\widehat{f}_{p+1}(\nu)=\al(\widehat{f}_p(\nu))$ modulo $\nu^{p+1}$,
  \item $S(\widehat{f}(\nu))=\widehat{f}(-\nu)$,
	\item the invertibility index of $\widehat{f}_p(\nu)$ is equal to $0$ for all $p \in {\mathbb N}$,
\end{enumerate}
where $\al$, $\vee$ and $S$ are the structure operations of ${\mathds{T}}/{\mathcal I} $ extended by $\mathbb{R}[[\nu]]$-linearity  to ${\mathds{T}}/{\mathcal I} [[\nu]]$.
\end{prop}
\begin{proof}
The first assertion follows from the second item of  Lemma \ref{lem:coproduct-explicit}. The second one is obtained by considering the definition \eqref{ChapFnu} and the relation $ \lfloor e^{i}\rfloor_{p+1} =\al (\lfloor e^{i}\rfloor_p)$. The third one is a consequence of the third item of Lemma \ref{lem:coproduct-explicit}.
Let us prove the last assertion. Proposition \ref{prop:indifference} implies that
$$\widehat{f}_p(\nu) =a_0 \mathds{1}+ \sum_{i \geq 1}^p a_i \nu^i  \lfloor e^i \rfloor_p$$
can be represented by ferns. The conclusion then follows from Proposition \ref{prop:invert-index-ferns}.
\end{proof}



\section{A Hom-group integrating a Hom-Lie algebra}

In this section, we aim to associate to any Hom-Lie algebra a Hom-group.
This construction uses  the study of the universal enveloping algebra
and elements of group-like type.

\subsection{Group-like elements in the free Hom-associative algebra with $1$-generator}

For $g(\nu)= {\mathds 1}+\sum_{i=1}^\infty {g_i\nu^i}$ a formal group-like element, $g_1$ is primitive. Unlike for  the free associative algebra with $1$-generator, not any primitive element of ${\mathds{T}}/{\mathcal I}$  could be the first order element of a  formal group-like element with $g_0={\mathds 1}$.

\begin{prop}\label{prop:negativeGroupLike}
The $(\al,\id)$-Hom-Hopf algebra ${\mathds{T}}/{\mathcal I}[[\nu]]$ does not admit  formal group-like elements
of the form $g(\nu)= {\mathds 1}+\sum_{i=1}^\infty {g_i\nu^i}$ where  $g_1$ is the leaf weighted 1-tree $\begin{tikzpicture}[baseline={([yshift=-.8ex]current bounding box.center)}]
\draw[line width=1pt] (0,0) -- (0,0.4);
\draw (0,0.6) node {$0$};
\end{tikzpicture}$.
\end{prop}
\begin{proof}
Assume that  $g(\nu)= {\mathds 1}+\sum_{i=1}^\infty {g_i\nu^i}$ is a formal group element. For example, the coefficient of $\nu^2$ in $\Delta (g(\nu))=g(\nu)\otimes g(\nu) $ yields
$$\Delta (g_2)= g_1\otimes g_1+ g_2\otimes {\mathds 1}+{\mathds 1}\otimes g_2.$$
This imposes that $g_2$ is in $B_2$ which is impossible, because the projection of $\Delta (g_2)$ on $B_1\otimes B_1$ is a linear combination of  elements of the form
\begin{tikzpicture}[xscale=0.4, yscale=0.4,baseline={([yshift=-.8ex]current bounding box.center)}]
\draw[line width=1pt] (0,-0.5) -- (0,1.5) ;
\draw (0,2) node {$k$};
\end{tikzpicture}
$\otimes$
\begin{tikzpicture}[xscale=0.4, yscale=0.4,baseline={([yshift=-.8ex]current bounding box.center)}]
\draw[line width=1pt] (0,-0.5) -- (0,1.5) ;
\draw (0,2) node {$l$};
\end{tikzpicture}
with $k$ or $l$ strictly positive.
\end{proof}
\begin{rem}
The proof of Proposition \ref{prop:negativeGroupLike} gives indeed that
 there is no $2$-order formal group-like elements of the form
$g(\nu)= {\mathds 1}+\sum_{i=1}^p {g_i\nu^i}$ where  $g_1$ is the leaf weighted 1-tree $\begin{tikzpicture}[baseline={([yshift=-.8ex]current bounding box.center)}]
\draw[line width=1pt] (0,0) -- (0,0.4);
\draw (0,0.6) node {$0$};
\end{tikzpicture}$.
\end{rem}

Let us show that formal group-like sequence is a relevant object by showing that the free Hom-associative algebra with $1$-generator admits a $1$-parameter family of formal group-like sequences, although it admits very few
group-like elements.

For all $s\in {\mathbb  R}$, consider the realization $\widehat{exp}(s)$ of the formal series $exp(s)=\sum_{i=0}^\infty{\frac{s^i}{i!}\nu^i}$. We   call the assignment $s \rightarrow \widehat{exp}(s)$  the \textbf{exponential sequence}.

For a better understanding of $\widehat{exp}(s)$, we give its first terms:
$$
\begin{array}{rcl}
\widehat{exp}_0(s)&=& {\mathds{ 1}} \\ \widehat{exp}_1(s) &=&{\mathds{ 1}} +  s \nu \begin{tikzpicture}[xscale=0.4, yscale=0.4,baseline={([yshift=-.8ex]current bounding box.center)}]
\draw[line width=1pt] (0,0) -- (0,2);
\draw (0,2.5) node {$0$};
\end{tikzpicture} \\
  \widehat{exp}_2(s) &=&{\mathds{ 1}} + s \nu \begin{tikzpicture}[xscale=0.4, yscale=0.4,baseline={([yshift=-.8ex]current bounding box.center)}]
\draw[line width=1pt] (0,0) -- (0,2);
\draw (0,2.5) node {$1$};
\end{tikzpicture}
+ \frac{s^2 \nu^2}{2!}
\begin{tikzpicture}[xscale=0.5, yscale=0.5,baseline={([yshift=-.8ex]current bounding box.center)}]
\draw[line width=1pt] (0,0) -- (0,1) -- (1,2);
\draw[line width=1pt] (0,1) -- (-1,2);
\draw (1,2.2) node {\scriptsize$0$};
\draw (-1,2.2) node {\scriptsize$0$};
\end{tikzpicture} \\
\widehat{exp}_3(s) &=&{\mathds{ 1}} + s \nu \begin{tikzpicture}[xscale=0.4, yscale=0.4,baseline={([yshift=-.8ex]current bounding box.center)}]
\draw[line width=1pt] (0,0) -- (0,2);
\draw (0,2.5) node {$2$};
\end{tikzpicture}
+ \frac{s^2 \nu^2}{2!}
\begin{tikzpicture}[xscale=0.5, yscale=0.5,baseline={([yshift=-.8ex]current bounding box.center)}]
\draw[line width=1pt] (0,0) -- (0,1) -- (1,2);
\draw[line width=1pt] (0,1) -- (-1,2);
\draw (1,2.2) node {\scriptsize$1$};
\draw (-1,2.2) node {\scriptsize$1$};
\end{tikzpicture}
+\frac{s^3 \nu^3}{3!}
 \begin{tikzpicture}[xscale=1, yscale=1,baseline={([yshift=-.8ex]current bounding box.center)}]
\draw[line width=1pt] (0,0) -- (0,1) -- (1,2);
\draw[line width=1pt] (0.5,1.5) -- (0,2);
\draw[line width=1pt] (0,1) -- (-1,2);
\draw (1,2.1) node {\scriptsize$0$};
\draw (0,2.1) node {\scriptsize$ 0$};
\draw (-1,2.1) node {\scriptsize$ 1$};
\end{tikzpicture}.
\end{array}
$$
Generally, we have
$$\widehat{exp}_p(s) =\mathds{1}+ \sum_{i = 1}^p \frac{ s^i}{i!} \nu^i   \lfloor e^i \rfloor_p .
$$

\begin{thm}
\label{theo:group-like_order_kappa}
 The exponential sequence $s \rightarrow \widehat{exp}(s)$ is  valued in
 the Hom-group $G_{seq}({\mathds T}/{\mathcal I}) $ of formal group-like sequence.\\
 Moreover,
 \begin{enumerate}
  \item  $ \widehat{exp}(0) $ is the unit element of the Hom-group of formal group-like sequences.
  \item  For all $s,t \in {\mathbb R}, k \in {\mathbb N}$,
     $ \widehat{exp}(s) \vee  \widehat{exp}(t) = \al(\widehat{exp}(s+t)) $.
  \item  $ S(\widehat{exp}(s)) = \widehat{exp}(-s)$ is a strict inverse, i.e.
   $$ \widehat{exp}(s) \vee  \widehat{exp}(-s) =   \widehat{exp}(-s) \vee \widehat{exp}(s) = \mathds{1} . $$
 \end{enumerate}
\end{thm}
\begin{proof}
The first item just follows from the definition of the realization of the formal series $f(\nu)=e^{0\nu}=1$.
The second item follows from the first item in Proposition \ref{Prop:proprietiesWidehat} applied to the classical relation $e^{s \nu} e^{t \nu} = e^{(s+t)\nu}$. The third item follows from the third item in Proposition \ref{Prop:proprietiesWidehat}.

It remains to show that the exponential sequence is  valued in formal group-like sequence, defined in item (iii)
of Definition \ref{def:groupLikeAndTheLikes}. The fourth item in Proposition \ref{Prop:proprietiesWidehat} implies that $\widehat{exp}_p(s)$ (i.e. the $p$-th term in the sequence $\widehat{exp}(s)$) has invertibility index equal to $0$, so that condition c) holds.
The second item in Proposition \ref{Prop:proprietiesWidehat} implies item b) in
Definition \ref{def:groupLikeAndTheLikes}. We are left with the task of showing that
$\widehat{exp}_p(s)$ is a $p$-order group-like element. This follows from the following computation,
which is done modulo $\nu^{p+1}$:
\begin{eqnarray*} \Delta  (\widehat{exp}_p(s)) &=& \Delta(\mathds{1})+ \sum_{i = 1}^p \frac{ s^i}{i!} \nu^i  \Delta \lfloor e^i \rfloor_p \\
&=&  \mathds{1} \otimes \mathds{1}+ \sum_{i = 1}^p \sum_{j = 1}^i \frac{ s^i}{i!} \nu^i  \left( \begin{matrix}
  i   \\
j
\end{matrix}\right)  \lfloor e^{j} \rfloor_p \otimes
 \lfloor e^{i-j} \rfloor_p \\
&=&  \widehat{exp}_p(s) \otimes  \widehat{exp}_p(s),
\end{eqnarray*}
where the first item of Lemma \ref{lem:coproduct-explicit} was used to go from the first to the second line.
\end{proof}

\subsection{Formal group-like sequences of the universal enveloping algebra}

Now, we define the exponential map for a Hom-Lie algebra $(\gg,[\cdot,\cdot],\al)$, in order to achieve a construction of a functor from the category of Hom-Lie algebras to the category of Hom-groups.
 For all $x_1, \dots,x_i \in \gg$, and  $p,i \in \mathbb{N}$ with $ p \geq i $,
define an element in ${\mathcal U}\gg$ by:
$$ \lfloor x_1, \dots,x_i \rfloor_p = e^i_p \otimes (x_1 \otimes \dots \otimes x_i). $$
If $x_1= \dots=x_i=x$, then we denote this product as $\lfloor x^i\rfloor_p $.
For $f(\nu) = \sum a_i \nu^i$  a formal series, we call \textbf{realization of $f(\nu)$ evaluated at $x$} the sequence
 $$\left(\sum_{i=0}^p \frac{s^i}{i!} \lfloor x^i \rfloor_p \right)_{p \in {\mathbb N}}.$$
For $f=exp(\nu)$ in particular, we define
the exponential map $\widehat{exp}(sx)$ to be the sequence in ${\mathcal U}\gg[[\nu]]$
obtained by taking the realization evaluated at $x$ of the formal series $e^{s\nu}$.
By construction, $\widehat{exp}(sx)$ is obtained by applying the Schur construction to
$\widehat{exp}(s)$ and to the element $x$, and the following theorem
can be derived easily from Theorem \ref{theo:group-like_order_kappa}.

\begin{thm}
\label{theo:group-like_x}
 For all $x \in {\gg} $ the exponential sequence $s \rightarrow \widehat{exp}(sx)$ is  valued in
 the Hom-group $G_{seq}(\gg) $. Moreover,
 \begin{enumerate}
  \item  $ \widehat{exp}(0 x) $ is the unit element $\mathds{1} \in G_{seq}(\gg) $.
  \item  For all $s,t \in {\mathbb R}, k \in {\mathbb N}$,
     $ \widehat{exp}(sx) \vee  \widehat{exp}(tx) = \al(\widehat{exp}((s+t)x)) = \widehat{exp}((s+t)\al(x))$.
  \item  $ S(\widehat{exp}(sx)) = \widehat{exp}(-sx)$ is a strict inverse, i.e.
   $$ \widehat{exp}(sx) \vee  \widehat{exp}(-sx) =   \widehat{exp}(-sx) \vee \widehat{exp}(sx) = \mathds{1} . $$
 \end{enumerate}
\end{thm}

For a better understanding of $\widehat{exp}(sx)$, we give its first terms:
$$
\begin{array}{rcl}
\widehat{exp}_0(sx)&=& {\mathds{ 1}} \\ \widehat{exp}_1(sx) &=&{\mathds{ 1}} +  s \nu \begin{tikzpicture}[xscale=0.4, yscale=0.4,baseline={([yshift=-.8ex]current bounding box.center)}]
\draw[line width=1pt] (0,0) -- (0,2);
\draw (0,2.5) node {$x$};
\end{tikzpicture} \\
  \widehat{exp}_2(sx) &=&{\mathds{ 1}} + s \nu \begin{tikzpicture}[xscale=0.4, yscale=0.4,baseline={([yshift=-.8ex]current bounding box.center)}]
\draw[line width=1pt] (0,0) -- (0,2);
\draw (0,2.5) node {$\al(x)$};
\end{tikzpicture}
+ \frac{s^2 \nu^2}{2!}
\begin{tikzpicture}[xscale=0.5, yscale=0.5,baseline={([yshift=-.8ex]current bounding box.center)}]
\draw[line width=1pt] (0,0) -- (0,1) -- (1,2);
\draw[line width=1pt] (0,1) -- (-1,2);
\draw (1,2.2) node {\scriptsize$x$};
\draw (-1,2.2) node {\scriptsize$x$};
\end{tikzpicture} \\
\widehat{exp}_3(sx) &=&{\mathds{ 1}} + s \nu \begin{tikzpicture}[xscale=0.4, yscale=0.4,baseline={([yshift=-.8ex]current bounding box.center)}]
\draw[line width=1pt] (0,0) -- (0,2);
\draw (0,2.5) node {$\al^2(x)$};
\end{tikzpicture}
+ \frac{s^2 \nu^2}{2!}
\begin{tikzpicture}[xscale=0.5, yscale=0.5,baseline={([yshift=-.8ex]current bounding box.center)}]
\draw[line width=1pt] (0,0) -- (0,1) -- (1,2);
\draw[line width=1pt] (0,1) -- (-1,2);
\draw (1,2.2) node {\scriptsize$\al(x)$};
\draw (-1,2.2) node {\scriptsize$\al(x)$};
\end{tikzpicture}
+\frac{s^3 \nu^3}{3!}
 \begin{tikzpicture}[xscale=1, yscale=1,baseline={([yshift=-.8ex]current bounding box.center)}]
\draw[line width=1pt] (0,0) -- (0,1) -- (1,2);
\draw[line width=1pt] (0.5,1.5) -- (0,2);
\draw[line width=1pt] (0,1) -- (-1,2);
\draw (1,2.1) node {\scriptsize$x$};
\draw (0,2.1) node {\scriptsize$ x$};
\draw (-1,2.1) node {\scriptsize$ \al(x)$};
\end{tikzpicture}.
\end{array}
$$

An immediate consequence of the expression of $ \widehat{exp}_1(sx) $ is the next proposition, that we invite the reader to see as saying that $G_{seq}(\gg)$ is large enough to be meaningful.
\begin{prop}\label{prop:expinjective}
For every $s \neq 0$, the assignment
$x \mapsto \widehat{exp}(sx)$
is an injection from $\gg$ to the Hom-group of formal group-like sequences $G_{seq}(\gg)$.
\end{prop}

In order to integrate a Hom-Lie algebra into a Hom-group, we compose the following two functors:
\begin{enumerate}
\item The functor $ {\mathcal U} $ from the  category of Hom-Lie algebras to the category of
Hom-Hopf algebras which consists in assigning  to a Hom-Lie algebra $\gg$
its universal algebra ${\mathcal U}\gg$, as in Remark \ref{rem:functor}.
\item The functor $G_{seq}$ from the category of Hom-Hopf algebras to the category of Hom-groups,
which consists in assigning  to a Hom-Hopf algebra $A$ its formal group-like sequences,
as in Remark \ref{rem:functorgroup}.
\end{enumerate}
Therefore the composition of these functors is a functor ${\mathfrak G}$ from the category of Hom-Lie algebras to the category of Hom-groups. Proposition \ref{prop:expinjective} implies that this functor is not trivial. Notice that it is compatible with the exponential map in the sense that  for every morphism of Hom-Lie algebra $\varphi:  \gg \to \gg'$,
the following diagram
 $$ \xymatrix{ \gg \ar[r]^{\varphi} \ar[d]^{  \widehat{exp}(s \cdot) }&  \gg' \ar[d]^{  \widehat{exp}(s \cdot) }\\ {\mathfrak G} (\gg) \ar[r]^{{\mathfrak  G}(\varphi)}& {\mathfrak G}(\gg').} $$
is commutative.

\end{document}